\PassOptionsToPackage{expansion=false}{microtype}
\pdfoutput=1
\documentclass{amsart}

\usepackage[mathscr]{eucal}
\usepackage{amssymb}
\usepackage[usenames,dvipsnames]{xcolor} 
\usepackage[normalem]{ulem}
\usepackage{wasysym}
\usepackage{amsthm}
\usepackage{bbold}
\usepackage{comment}
\usepackage{enumitem}
\usepackage{amsmath}
\usepackage{tikz-cd}
\usepackage{stmaryrd} 
\usepackage{etoolbox} 
\usepackage{microtype}
\usepackage{adjustbox}
\usepackage{ mathdots }
\usepackage{bbm}
\usepackage{mathtools}

\usepackage[unicode]{hyperref} 

\definecolor{dark-red}{rgb}{0.5,0.15,0.15}
\definecolor{dark-blue}{rgb}{0.15,0.15,0.6}
\definecolor{dark-green}{rgb}{0.15,0.6,0.15}

\hypersetup{
    colorlinks, linkcolor=OliveGreen,
    citecolor=OliveGreen, urlcolor=OliveGreen
}

\usepackage[nameinlink,capitalise,noabbrev]{cleveref}

\numberwithin{equation}{section}
\usepackage{caption}

\newtheorem{thmx}{Theorem}

\newtheorem{Thm}[equation]{Theorem}
\newtheorem*{Thm*}{Theorem}
\newtheorem*{MainThm*}{Main Theorem}
\newtheorem{Prop}[equation]{Proposition}
\newtheorem{Lem}[equation]{Lemma}
\newtheorem{Cor}[equation]{Corollary}

\newtheorem*{Que*}{Question}

\theoremstyle{remark}
\newtheorem{Def}[equation]{Definition}

\newtheorem{Exa}[equation]{Example}
\newtheorem{Exas}[equation]{Examples}
\newtheorem{Cons}[equation]{Construction}

\newtheorem{Hyp}[equation]{Hypothesis}
\newtheorem{Rec}[equation]{Recollection}

\newtheorem{Rem}[equation]{Remark}

\tikzset{
    labelrotatebelow/.style={anchor=north, rotate=90, inner sep=1.0mm}
}
\tikzset{
    labelrotateabove/.style={anchor=south, rotate=90, inner sep=1.0mm}
}

\newcommand{\nc}{\newcommand}
\nc{\dmo}{\DeclareMathOperator}
\renewcommand{\emptyset}{\varnothing}

\usepackage[disable]{todonotes}

\nc{\overbar}[1]{\mkern 1.5mu\overline{\mkern-1.5mu#1\mkern-1.5mu}\mkern 1.5mu}

\nc{\kappaaux}{g}
\nc{\kappaCh}{{\kappaaux(\cat C_h)}}
\nc{\kappam}{{\kappaaux({\mathfrak m})}}
\nc{\kappaP}{\Gamma_{\cat P}\unit}
\nc{\kappaQ}{{\kappaaux(\cat Q)}}
\nc{\kappaCP}{{\kappaaux_{\cat C}(\cat P)}}
\nc{\kappaDP}{{\kappaaux_{\cat D}(\cat P)}}
\nc{\kappaCQ}{{\kappaaux_{\cat C}(\cat Q)}}
\nc{\kappaDQ}{{\kappaaux_{\cat D}(\cat Q)}}
\nc{\kappaphiB}{{\kappaaux(\phi(\cat B))}}
\nc{\kappaphiQ}{{\kappaaux(\varphi(\cat Q))}}
\newcommand{\noloc}{\;\mathord{:}\,}
\dmo{\Sub}{Sub}
\dmo{\Proj}{Proj}
\dmo{\LMod}{LMod}
\dmo{\cell}{cell}
\nc{\Prst}{{\cat P}\mathrm{r^{st}}}
\nc{\Mack}[2]{\mathrm{Mack}_{#1}(#2)}
\dmo{\fin}{{fin}}
\dmo{\Sphere}{\mathbb{S}}
\dmo{\Alg}{Alg}
\dmo{\CAlg}{CAlg}
\nc{\HA}{{\rmH \hspace{-0.2em}\bbA}}
\nc{\HZ}{{\rmH \hspace{-0.2em}\bbZ}}
\nc{\HZbar}{{\rmH \hspace{-0.2em}\underline{\bbZ}}}
\nc{\Fp}{{\bbF_{\hspace{-0.1em}p}}}
\nc{\HFp}{{\rmH \hspace{-0.15em}\bbF_{\hspace{-0.1em}p}}}

\nc{\mathfrakp}{\mathfrak{p}}
\nc{\mathfrakq}{\mathfrak{q}}
\nc{\mathfrakS}{\mathfrak{S}}
\nc{\mathfrakT}{\mathfrak{T}}
\nc{\Z}{\mathbb{Z}}
\nc{\cF}{\mathcal{F}}
\nc{\hspec}[1]{\Spc^\mathrm{h}({#1})}

\dmo{\Id}{Id}
\dmo{\Loc}{Loc}
\dmo{\Spc}{Spc}
\nc{\thick}[1]{\mathrm{thick}\langle #1 \rangle}
\nc{\thickt}[1]{\mathrm{thick}_\otimes\langle #1 \rangle}
\dmo{\End}{End}
\dmo{\Mor}{Mor}
\dmo{\Hom}{Hom}
\dmo{\id}{id}
\dmo{\im}{im}
\dmo{\Ker}{Ker}
\dmo{\ind}{ind}
\dmo{\Ind}{Ind}
\dmo{\CoInd}{coind}
\dmo{\dg}{dg}
\dmo{\res}{res}
\dmo{\infl}{infl}
\dmo{\triv}{triv}
\dmo{\Tel}{Tel} 
\dmo{\grMod}{grMod}%
\dmo{\Mod}{Mod}%
\dmo{\opname}{op}
\dmo{\SH}{\mathcal{S}\mathcal{H}}
\dmo{\smallb}{b}
\dmo{\Spec}{Spec}
\dmo{\supp}{supp}
\dmo{\Supp}{Supp}
\dmo{\cosupp}{cosupp}
\dmo{\Cosupp}{Cosupp}
\dmo{\hsupp}{hsupp}
\dmo{\esupp}{esupp}
\dmo{\qsupp}{qsupp}
\dmo{\eqsupp}{eqsupp}
\dmo{\Pinj}{Pinj}
\dmo{\Inj}{Inj}
\dmo{\Arr}{Arr}

\nc{\bbL}{\mathbb{L}}
\nc{\bbA}{\mathbb{A}}
\nc{\bbE}{\mathbb{E}}
\nc{\bbN}{\mathbb{N}}
\nc{\bbQ}{\mathbb{Q}}
\nc{\bbZ}{\mathbb{Z}}
\nc{\bbF}{\mathbb{F}}
\nc{\bbS}{\mathbb{S}}
\nc{\cat}[1]{\mathscr{#1}}

\nc{\cA}{\mathcal{A}}
\nc{\cB}{\mathcal{B}}
\nc{\cC}{\mathcal{C}}
\nc{\cD}{\mathcal{D}}
\nc{\cE}{\mathcal{E}}
\nc{\cU}{\mathcal{U}}

\nc{\CB}{\mathbf{B}}

\nc{\sD}{\mathsf{D}}

\nc{\ihom}{{\underline{\hom}}}
\nc{\iHom}{\mathcal{H}\mathrm{om}}

\nc{\Mid}{\,\big|\,}
\nc{\SET}[2]{\big\{\,#1\Mid#2\,\big\}}
\nc{\unit}{\mathbb{1}}
\nc{\xra}{\xrightarrow}

\dmo{\Sp}{Sp}
\dmo{\Ho}{Ho}
\dmo{\Fin}{Fin}
\dmo{\add}{add}
\dmo{\Fun}{Fun}
\dmo{\Ext}{Ext}

\dmo{\Map}{Map}
\dmo{\Span}{Span}
\dmo{\N}{N}
\dmo{\Cat}{Cat}
\dmo{\colim}{colim}
\dmo{\hocolim}{hocolim}
\dmo{\Ch}{Ch}
\dmo{\Gr}{Gr}
\nc{\CA}{\cat A}
\nc{\CU}{\cat U}

\dmo{\Ab}{Ab}
\dmo{\Set}{Set}
\dmo{\ev}{ev}
\dmo{\Spcl}{Spcl}
\nc{\Funadd}{\Fun_{\add}}
\dmo{\proj}{proj}
\dmo{\deftensor}{Def^\otimes}
\dmo{\defcoid}{Def^{coid}}

\dmo{\dual}{dual}

\dmo{\Perf}{Perf}
\dmo{\tel}{tel}

\dmo{\rk}{rk}

\dmo{\FIdl}{\mathrm{Thick}_\otimes^{\mathrm{fg}}}
\dmo{\Idl}{\mathrm{Thick}_\otimes}
\dmo{\PIdl}{Spec}
\dmo{\Fd}{Fd}
\nc{\hatS}{\widehat{\mathcal{S}}}
\nc{\CS}{\mathcal{S}}
\dmo{\LS}{LS}
\nc{\hatLS}{\widehat{\mathrm{LS}}}

\dmo{\glo}{glo}
\dmo{\Pic}{Pic}
\dmo{\gl}{gl}
\nc{\fan}[1]{\mathscr{F}_{#1}}
\dmo{\GL}{GL}
\dmo{\Out}{Out}
\dmo{\Fan}{Fan}
\dmo{\cons}{cons}

\renewcommand{\leq}{\leqslant}
\renewcommand{\geq}{\geqslant}

\nc{\ua}{\mathord{\uparrow}}

\newcommand{\sfD}{\mathsf{D}}

\DeclareMathOperator{\fg}{fg}
\DeclareMathOperator{\op}{op}

\newcommand{\A}[1]{\mathsf{A}(#1)} 
\newcommand{\K}[1]{\mathsf{K}(#1)} 
\newcommand{\D}[1]{\mathsf{D}(#1)} 
\renewcommand{\P}[1]{\mathsf{P}(#1)} 

\newcommand{\bbm}       {\left[\begin{matrix}}
\newcommand{\ebm}       {\end{matrix}\right]}
\newcommand{\bsm}       {\left[\begin{smallmatrix}}
\newcommand{\esm}       {\end{smallmatrix}\right]}
\newcommand{\Vect}      {\operatorname{Vect}}
\newcommand{\cof}       {\mathsf{cof}}
\newcommand{\fib}       {\mathsf{fib}}
\newcommand{\acf}       {\mathsf{acf}}
\newcommand{\afb}       {\mathsf{afb}}
\newcommand{\gacf}      {\mathsf{gacf}}
\newcommand{\gcof}      {\mathsf{gcof}}
\newcommand{\we}        {\mathsf{we}}
\newcommand{\cok}       {\operatorname{cok}}
\newcommand{\img}       {\operatorname{image}}
\newcommand{\tg}        {\widetilde{g}}
\newcommand{\ti}        {\widetilde{i}}
\newcommand{\tk}        {\widetilde{k}}
\newcommand{\tn}        {\widetilde{n}}
\newcommand{\tp}        {\widetilde{p}}
\newcommand{\tq}        {\widetilde{q}}

\newcommand{\Aut}       {\operatorname{Aut}}

\newcounter{enum-resume-hack}
\Crefname{Thm}{Theorem}{Theorems}
\Crefname{Prop}{Proposition}{Propositions}
\Crefname{Lem}{Lemma}{Lemmas}
\Crefname{thmx}{Theorem}{Theorems}
\begin{document}

\title[Global representation theory: Homological foundations]{Global representation theory: \\
\smaller{Homological foundations}}

\author[Barrero]{Miguel Barrero}
\author[Barthel]{Tobias Barthel}
\author[Pol]{Luca Pol}
\author[Strickland]{Neil Strickland}
\author[Williamson]{Jordan Williamson}

\date{\today}

\makeatletter
\patchcmd{\@setaddresses}{\indent}{\noindent}{}{}
\patchcmd{\@setaddresses}{\indent}{\noindent}{}{}
\patchcmd{\@setaddresses}{\indent}{\noindent}{}{}
\patchcmd{\@setaddresses}{\indent}{\noindent}{}{}
\makeatother

\address{}
\email{}
\urladdr{}

\address{Miguel Barrero, Department of Mathematics, University of Aberdeen, Fraser Noble Building, Aberdeen AB24 3UE, UK}
\email{miguel.barrero@abdn.ac.uk}
\urladdr{https://sites.google.com/view/mbarrero}

\address{Tobias Barthel, Max Planck Institute for Mathematics, Vivatsgasse 7, 53111 Bonn, Germany}
\email{tbarthel@mpim-bonn.mpg.de}
\urladdr{https://sites.google.com/view/tobiasbarthel/}

\address{Luca Pol, Fakult{\"a}t f{\"u}r Mathematik, Universit{\"a}t Regensburg, Universit{\"a}tsstraße 31, 93053 Regensburg, Germany}
\email{luca.pol@mathematik.uni-regensburg.de}
\urladdr{https://sites.google.com/view/lucapol/}

\address{Neil P. Strickland,
 School of Mathematical and Physical Sciences,
 Hicks Building, 
 Sheffield S3 7RH, 
 UK
}
\email{N.P.Strickland@sheffield.ac.uk}
\urladdr{https://strickland1.org}

\address{Jordan Williamson, Department of Algebra, Faculty of Mathematics and Physics, \linebreak Charles University in Prague, Sokolovsk\'{a} 83, 186 75 Praha, Czech Republic}
\email{williamson@karlin.mff.cuni.cz}
\urladdr{https://jordanwilliamson1.github.io/}

\begin{abstract}
A global representation is a compatible collection of representations of the outer automorphism groups of the groups belonging to some collection of finite groups $\mathscr{U}$. Global representations assemble into an abelian category $\mathsf{A}(\mathscr{U})$, simultaneously generalising classical representation theory and the category of VI-modules appearing in the representation theory of the general linear groups. In this paper we establish homological foundations of its derived category $\mathsf{D}(\mathscr{U})$. We prove that any complex of projective global representations is DG-projective, and hence conclude that the derived category admits an explicit model as the homotopy category of projective global representations. We show that from a tensor-triangular perspective it exhibits some unusual features: for example, there are very few dualizable objects and in general many more compact objects. Under more restrictive conditions on the family $\mathscr{U}$, we then construct torsion-free classes for global representations which encode certain growth properties in $\mathscr{U}$. This lays the foundations for a detailed study of the tensor-triangular geometry of derived global representations which we pursue in forthcoming work. 
\vspace{-0.7cm}
\end{abstract}

\subjclass[2020]{18G80, 20C99, 20J05; 18A25}


\maketitle
\setcounter{tocdepth}{1}
\tableofcontents
\vspace{-0.7cm}

\section{Introduction}
This is the first in a series of papers which studies the derived categories of compatible representations of outer automorphism groups. Here we lay the homological foundations for this study, by giving an explicit model for these categories, by constructing an efficient model for each of its objects, and by characterizing their compact and dualizable objects. Moreover, under additional assumptions, we construct certain torsion-free elements that will play an important role in a sequel to this paper.

More precisely, we consider a full, replete subcategory $\cat{U}$ of the category of finite groups and conjugacy classes of surjective group homomorphisms. The category of $\cat{U}$-global representations is the abelian category of functors 
    \[
        \A{\cat{U}} \coloneqq \Fun(\cat{U}^{\op}, \Mod{k})
    \] 
for a fixed field $k$ of characteristic 0. Since the endomorphism group of an object $G$ in $\cat{U}$ is the outer automorphism group $\Out(G)$, one sees that a global representation is a collection of $\Out(G)$-representations for each $G \in \cat{U}$, together with compatibility conditions governed by the functoriality. 

The study of global representations is of interest within pure algebra, but also features prominently in other areas:
\begin{enumerate}
    \item \emph{Representation theory:} In the special case when $\cat{U}$ consists of a single group $G$, we have $\A{\cat{U}} = \Mod{k[\Out(G)]}$, the category of $\Out(G)$-representations. Since any finite group can be realised as the outer automorphism group of a finite group~\cite{MO, realizingout}, the category of global representations therefore generalises the classical representation theory of finite groups, and provides a framework for studying compatible collections of representations.
    \item \emph{Representation stability:} The category of global representations is closely related to, and generalises, categories related to representation stability appearing in the literature such as VI-modules~\cite{PutmanSam2017, Nagpal} and FI-modules \cite{FI1, FI2}. For example, if $\cat{U}$ is the collection of elementary abelian $p$-groups, then a $\cat{U}$-global representation is
    a diagram of representations of the general linear groups:
\[
\begin{tikzcd}
   X(1)      \arrow[r] \arrow[out=60,in=120,loop,distance=2em,"1"'] & 
   X(C_p)    \arrow[r, yshift=2.1mm] \arrow[r, yshift=-2.1mm,"\vdots"] \arrow[out=60,in=120,loop,distance=2em,"GL_1(\mathbb{F}_p)"'] & 
   X(C_p^2)  \arrow[r, yshift=2.5mm] \arrow[r, yshift=-2.5mm,"\vdots" {yshift=0.4mm}] \arrow[out=60,in=120,loop,distance=2em,"GL_2(\mathbb{F}_p)"'] & 
   X(C_p^3)  \arrow[r, yshift=2.9mm] \arrow[r, yshift=-2.9mm,"\vdots" {yshift=0.8mm}] \arrow[out=60,in=120,loop,distance=2em,"GL_3(\mathbb{F}_p)"'] & 
   \dotsb
  \end{tikzcd}
\]
    where the horizontal maps are induced by the projections.
    In this case, $\A{\cat{U}}$ is equivalent to the category of VI-modules by Pontryagin duality.
    \item \emph{Equivariant homotopy theory:} Global cohomology theories  consist of a compatible collection of equivariant cohomology theories, one for each group in a given family. These assemble into a category of \emph{global spectra} as introduced by Schwede~\cite{Schwedebook}. There is a deep connection between global representation theory and global homotopy theory: there is a tensor-triangulated equivalence
    \[\mathrm{Sp}_{\cat{U}\text{-}\glo}^\bbQ \simeq \D{\cat{U}}\]
    between the category of rational global spectra and the derived category $\D{\cat{U}}$ of $\A{\cat{U}}$ for any global family of finite groups $\cat{U}$~\cite{Schwedebook, Wimmerthesis, BBPSWttgeometry}. As such the results of this paper may be reinterpreted in the context of global equivariant homotopy theory. 
\end{enumerate}

In~\cite{PolStrickland2022}, Pol and Strickland undertook a detailed study of the abelian category $\A{\cat{U}}$, which exhibits various unusual properties. For example, under mild assumptions on $\cat{U}$, the projective objects form a subcategory of the injective objects. In this paper, we investigate its derived category $\D{\cat{U}}$ in more detail. This is a triangulated category, but the pointwise tensor product of representations also equips $\D{\cat{U}}$ with a compatible monoidal structure, so that it is a \emph{tensor-triangulated category}. We show that it has a rich structure and many surprising and uncommon features. 

We now turn to describing the results of this paper in more detail.

\subsection*{Projectives and modelling the derived category}
The first key result of this paper is a description of the DG-projective complexes in $\A{\cat U}$ which in turn yields a simple description of the derived category $\D{\cat{U}}$. Recall that DG-projectivity is the appropriate analogue of projectivity of modules for complexes: indeed, left derived functors on complexes are defined by taking DG-projective resolutions. Any DG-projective complex is a complex of projectives, but the converse is false in general. Our first main result (\cref{thm-cofibrant}) shows that for our categories of interest $\A{\cat{U}}$, the converse is always true which in particular makes the computation of derived functors simple.
\begin{thmx}
 Let $X$ be a chain complex of projective objects in $\A{\cat{U}}$, and let $Y$ be any acyclic complex;
 then any chain map $X\to Y$ is nullhomotopic.  In other words, every chain complex of projectives 
 is DG-projective. 
\end{thmx}
Consequently, we obtain the following concrete model for the derived category $\D{\cat{U}}$ in \cref{thm-derived-category-equal-K-proj}:
\begin{thmx}
    There is a non-unital symmetric monoidal equivalence of categories 
        \[
            \D{\cat{U}} \simeq \K{\P{\cat{U}}}
        \]
    between $\D{\cat{U}}$ and the homotopy category of projectives in $\A{\cat{U}}$. Moreover, if the tensor unit of $\A{\cat{U}}$ is projective, then this equivalence refines to a tensor-triangulated equivalence.
\end{thmx}

One application of the above results is the construction of a convenient replacement for complexes, which provides an efficient means of transporting information between the abelian category $\A{\cat{U}}$ and its derived category $\D{\cat{U}}$. By restricting to the subcategories $\cat{U}_{\leq s}$ of groups in $\cat{U}$ of order less than or equal to $s$, one obtains an ascending filtration of complexes 
    \[
        0 = L_{\leq 0}X \subseteq L_{\leq 1}X \subseteq L_{\leq 2}X \subseteq \ldots.
    \] 
The corresponding subquotients $L_sX = L_{\leq s}X/L_{\leq s-1}X$ are in $\D{\cat{U}_{=s}}$, where $\cat{U}_{=s}$ is the subcategory of $\cat{U}$ of groups of order precisely $s$. Since $\cat{U}_{=s}$ is a groupoid, all objects in $\A{\cat{U}_{=s}}$ are projective, so the homological algebra in the subquotients is substantially simplified. We say that a complex of projectives is \emph{thin} if the quotient complexes $L_sX$ have zero differential. We prove that any complex is quasi-isomorphic to a thin complex, and moreover, that this \emph{thin replacement} is unique up to isomorphism of complexes.  We refer the reader to \cref{sec-thin} for precise statements of the existence, and structure, of these thin replacements.

Using the properties of thin replacements, we prove the following structural result for the homology of compact objects in $\D{\cat{U}}$, see \cref{thm-compact-torsion-free-element}. This result is one key input into the study of the tensor-triangular geometry of $\D{\cat{U}}$ which we undertake in the sequel~\cite{BBPSWttgeometry}.
\begin{thmx}
    Suppose that $\cat{U}$ is multiplicative global, i.e., closed under subgroups, quotients, and finite products, and let $X \in \D{\cat{U}}$ be a non-zero, compact object. Then the homology of $X$ contains a torsion-free element $x \in H_*(X)$, in the sense that, for all maps $\alpha \in \cat{U}$, we have $\alpha^*(x) \neq 0$.
\end{thmx}
The existence of such torsion-free elements captures a subtle feature of $\D{\cat{U}}$, providing a qualitative measure for the growth of the number of surjections in $\cat{U}$ that may fail for non-multiplicative families. 

\subsection*{The structure of compact and dualizable objects}
Given a triangulated category with set-indexed coproducts such as $\D{\cat{U}}$, one seeks to understand the building blocks of the category. In any triangulated category $\cat{T}$, a distinguished role is played by the subcategory $\cat{T}^c$ of \emph{compact} objects. If there is also a compatible monoidal structure so that the category is \emph{tensor-triangulated}, one may also consider the full subcategory of \emph{dualizable} objects $\cat{T}^{\dual}$. 

In nature, many large tensor-triangulated categories are compactly
generated, so that every object may be constructed from the compact
objects using triangles and sums. In fact, often even more is
true and the compact and the dualizable objects coincide, in which
case the category is said to be \emph{rigidly-compactly
 generated}. For instance, this holds for the derived categories of
commutative rings, for the stable module categories of finite groups,
and for the stable homotopy category of spectra.  It also holds for
the equivariant stable homotopy category of $G$-spectra for any single
finite group $G$, but we show that it usually fails for the derived categories of global representations that are of interest in this paper. Combining \cref{prop-comp-gen}, \cref{dualisthick1}, and \cref{groupoidrigid}, we show:

\begin{thmx}
    For any category of finite groups $\cat{U}$, the category $\D{\cat{U}}$ is compactly generated, and it is rigidly-compactly generated if and only if $\cat U$ is a finite groupoid. 
    If $\cat{U}$ contains the trivial group, we have
        \begin{equation}\label{eq:inclusion}
            \thick{\unit} = \D{\cat{U}}^{\dual} \subseteq \D{\cat{U}}^c.
        \end{equation}
\end{thmx}

In fact, we exhibit a distinguished collection of compact generators for $\cat{U}$, given by projective generators of the abelian category $\A{\cat{U}}$. Moreover, we establish various equivalent conditions for an object of $\D{\cat{U}}$ to be compact, see \cref{prop-perfect}, which build on the structure of thin replacements as described above. In particular, we deduce that $\D{\cat{U}}^c$ is closed under tensor products in $\D{\cat{U}}$ and, as such, is itself tensor-triangulated.

We note that the above results demonstrate exotic behaviour even compared to other non-rigidly compactly generated tensor-triangulated categories. For example, in the category of $K(n)$-local spectra or the derived category of $kG$-modules, every compact object is dualizable, but not conversely. In particular, the strictness of \eqref{eq:inclusion} adds a substantial new layer of complexity to the study of the tensor-triangular geometry of $\D{\cat{U}}^c$, as we discuss in the sequel. 

\subsection*{The projective model structure}
Finally, we construct a convenient model structure on $\D{\cat{U}}$, see \cref{thm-proj-model-structure,prop-cof-gen}. 

\begin{thmx}
    There is a projective model structure on $\D{\cat{U}}$ which is proper, cofibrantly generated, and monoidal. The weak equivalences are given by the quasi-isomorphisms, while the fibrations are the epimorphisms. 
\end{thmx}

While the existence of such a model structure could also be established via more abstract techniques, we view our construction as a useful proof of concept that may be of independent interest. 

\subsection*{Conventions}

Unless stated otherwise, throughout this paper we will work over a field $k$ of characteristic $0$, often omitted from the notation. We use $\cat{U}$ to denote a replete full subcategory of the category $\cat{G}$ of finite groups and conjugacy classes of surjective group homomorphisms, usually assumed to be widely closed as recorded in \cref{hyp-basic}. 

\subsection*{Acknowledgements}
MB is supported by the EPSRC grant EP/X038424/1 “Classifying spaces, proper actions and stable homotopy theory”. TB is supported by the European Research Council (ERC) under Horizon Europe (grant No.~101042990). LP is supported by the SFB 1085 Higher Invariants in Regensburg. JW is supported by the project PRIMUS/23/SCI/006 from Charles University and the Charles University Research Centre
program No. UNCE/24/SCI/022.

MB, TB, and LP are grateful to Max Planck Institute for Mathematics in Bonn for its hospitality and financial support. MB, TB, LP, and JW would also like to thank the Fondation des Treilles for its support and hospitality, where work on this paper was undertaken. The authors would also like to thank the Hausdorff Research Institute for Mathematics for its hospitality and support during the trimester program ‘Spectral Methods in Algebra, Geometry, and Topology’, funded by the Deutsche Forschungsgemeinschaft under Germany’s Excellence Strategy – EXC-2047/1 – 390685813. JW would like to thank the Isaac Newton Institute for Mathematical
Sciences for the support and hospitality during the programme `Topology, representation theory and higher 
structures' when work on this paper was undertaken. This work was supported by: EPSRC Grant Number EP/R014604/1. The authors would like to thank the Isaac Newton Institute for Mathematical Sciences, Cambridge, for support and hospitality during the programme `Equivariant homotopy theory in context', where work on this paper was undertaken. This work was supported by EPSRC grant EP/Z000580/1.

\section{Generalities in homological algebra}\label{sec-hom-alg}
In this section we introduce various categories associated to an abelian category, discuss their properties and the additional structures they canonically carry. We then study projective objects in the category of chain complexes and various features of abelian categories in which all objects are projective. 

\subsection{Graded objects, chain complexes and the derived category}
In this subsection we introduce the categories of graded objects and chain complexes in an additive category. After discussing some key properties they enjoy we introduce the homotopy category of chain complexes and the derived category of an abelian category. As these results are well-known we will not include details and instead use this as an opportunity to set some notation.

\begin{Def}\label{defn-Gr-Ch}
 For any additive category $\cat{B}$, we let $\Gr(\cat{B})$ denote the
 corresponding category of $\Z$-graded objects.  We define
 $\Sigma\colon \Gr(\cat{B})\to\Gr(\cat{B})$ by $(\Sigma X)_k=X_{k-1}$.  We write
 $\Ch(\cat{B})$ for the category of $\Z$-graded chain complexes, equipped
 with differentials $d\colon X_k\to X_{k-1}$.
\end{Def}

\begin{Rem}
There are fully faithful functors 
\[
\cat B \hookrightarrow \Gr(\cat B) \hookrightarrow \Ch(\cat B)
\]
which identify $\cat{B}$ with
$\{X\in\Gr(\cat{B})\mid X_n=0\text{ for all }n\neq 0\}$, and $\Gr(\cat{B})$
with $\{X\in\Ch(\cat{B})\mid d=0\colon X_n\to X_{n-1}\text{ for all $n$}\}$. 
\end{Rem}

In the next recollection we discuss the structure that the categories $\Gr(\cat B)$ and $\Ch(\cat B)$ inherit from $\cat B$.

\begin{Rec}\label{rec-internal_homs}
If $\cat{B}$ has a symmetric monoidal structure and coproducts, it induces
the same kind of structure on $\Gr(\cat{B})$ and $\Ch(\cat{B})$ in a
well-known way.  We will generically write $\otimes$ for the monoidal
product and $\unit$ for the unit object.

If $\cat{B}$ is a closed monoidal category, then we will write
$\iHom(X,Y)$ for the associated internal mapping objects.  Thus, for
$Y,Z\in\cat{B}$ we have $\iHom(Y,Z)\in\cat{B}$, and there are coherent natural
isomorphisms 
\[
\Hom_{\cat{B}}(X,\iHom(Y,Z))\cong\Hom_{\cat{B}}(X\otimes Y,Z) \quad\mathrm{in}\; \Ab,
\]
and
\[
\iHom(X,\iHom(Y,Z))\cong\iHom(X\otimes Y,Z)\quad \mathrm{in} \;\cat{B}.
\]
If $\cat{B}$ also has products, then there are induced internal mapping objects for
$\Gr(\cat{B})$ given by $\iHom(X,Y)_n=\prod_{i\in\Z}\iHom(X_{i-n},Y_i)$.
If $X$ and $Y$ have differentials then they induce a differential on
$\iHom(X,Y)$, and this makes $\Ch(\cat{B})$ into a closed symmetric
monoidal category as well.  On the other hand, we also have abelian
groups 
\[ \Hom(X,Y)_n=\prod_{i\in\Z}\Hom_{\cat{B}}(X_{i-n},Y_i) \cong 
    \Hom_{\cat{B}}(\unit,\iHom(X,Y)_n),
\]
which again form a chain complex. In summary, given $X,Y \in \Ch(\cat B)$, we have
\[
\Hom(X,Y) \in \Ch(\Ab) \quad \mathrm{and} \quad \iHom(X,Y)\in \Ch(\cat B).
\]
\end{Rec}

\begin{Rem}
    One verifies that the group of $0$-cycles $Z_0\Hom(X,Y)$ consists of chain maps $X\to Y$ and so can be identified with the morphism set $\Hom_{\Ch(\cat B)}(X,Y)$. The quotient group $H_0\Hom(X,Y)$ is the group of homotopy classes of chain maps.
\end{Rem}
\begin{Rem}
 If $\cat{B}$ is abelian and $Y\in\Ch(\cat{B})$ then we can define graded objects 
 $Z_*(Y)$ and $B_*(X)$ of cycles and boundaries.   Then for $X\in\Gr(\cat{B})$ we find that 
 \begin{align*}
  \Hom_{\Ch(\cat{B})}(X,Y) &\cong \Hom_{\Gr(\cat{B})}(X,Z_*(Y)) \\
  \Hom_{\Ch(\cat{B})}(Y,X) &\cong \Hom_{\Gr(\cat{B})}(Y/B_*(Y),X). 
 \end{align*}
\end{Rem}
We next introduce the homotopy category of chain complexes.
\begin{Def}\label{def-K(B)}
   For any additive category $\cat B$, we write $\K{\cat B}$ for the category whose objects are chain complexes in $\cat B$ and with hom sets given by 
   \[
   \Hom_{\K{\cat B}}(X,Y) \coloneqq H_0\Hom(X,Y)
   \]
   for all $X,Y \in \Ch(\cat B)$. If $\cat{B}$ has a closed symmetric monoidal structure and products and coproducts, then $\K{\cat{B}}$ also inherits these structures.
\end{Def}

\begin{Rem}\label{rem-triangulation}
 We will also need to use the fact that $\K{\cat{B}}$ has a canonical
 structure as a triangulated category.  For future use, we recall a
 few details of this.  For any morphism $f\colon X\to Y$ in $\Ch(\cat{B})$,
 there is a standard way to define a mapping cone $Cf\in\Ch(\cat{B})$ and
 a triangular diagram 
 \begin{equation}\label{triangle} 
 X \xrightarrow{f} Y \xrightarrow{i} Cf \xrightarrow{p} \Sigma X \xrightarrow{\Sigma f} \Sigma Y. 
 \end{equation}
 By definition, a distinguished triangle in $\K{\cat{B}}$ is a
 triangular diagram that is isomorphic to one that arises as above.
 However, there is also another standard source of distinguished
 triangles.  Suppose for simplicity that $\cat{B}$ is abelian, and that we
 have a short exact sequence $X\xrightarrow{f}Y\xrightarrow{g}Z$ in $\Ch(\cat{B})$.
 Suppose we also have maps $X\xleftarrow{u}Y\xleftarrow{v}Z$ in $\Gr(\cat{B})$ with
 $uf=1_X$ and $fu+vg=1_Y$ and $gv=1_Z$.  (If each $Z_n$ is projective,
 it will always be possible to choose such maps $u$ and $v$.)  One can
 then check that the formula $h=-udv$ gives a chain map $Z\to\Sigma X$,
 that the resulting sequence $X\xrightarrow{f}Y\xrightarrow{g}Z\xrightarrow{h}\Sigma X$ is a
 distinguished triangle, and that every distinguished triangle is
 isomorphic to one of this type.  This is proved as~\cite[tag 014L]{stacks} for example.  Therefore if $f\colon X\to Y$ is a
 monomorphism in $\Ch(\cat{B})$ with degreewise projective cokernel, then
 the cokernel is homotopy equivalent to the mapping cone. 
\end{Rem}

Finally, we introduce the derived category. 

\begin{Def}\label{def-derived-cat-abcat}
Let $\cat A$ be an abelian category. A chain map $f \colon X \to Y$ is a \emph{quasi-isomorphism} if the induced map in homology $H_*(f)\colon H_*(X) \xrightarrow{\sim} H_*(Y)$ is an isomorphism. The derived category $\D{\CA}$ of the abelian category $\CA$ is the localization of $\K{\CA}$ at the class of quasi-isomorphisms, that is 
 \[
 \D{\CA} \coloneqq \K{\CA}[\mathrm{q.iso}^{-1}].
 \]
\end{Def}
For the specific abelian category in which we are interested in this paper, we will give a more concrete description of the derived category in \cref{thm-derived-category-equal-K-proj}.

\begin{Rem}
    Recall that the derived category $\D{\CA}$ has a canonical triangulated category structure with distinguished triangles given by those triangular diagrams that are isomorphic to one of the form given in \eqref{triangle}. We also observe that if the abelian category $\cat A$ is closed symmetric monoidal, then by deriving the tensor product and internal hom functors we obtain the same structure on the derived category.
\end{Rem}
\subsection{Projective objects in chain complexes}
In this subsection we will discuss projective objects in the abelian category of chain complexes and prove some results about them which we will need in the body of the paper.

\begin{Def}\label{defn-Dl}
 Let $\cat B$ be an additive category. We define $\Delta\colon \Gr(\cat{B})\to\Ch(\cat{B})$ by 
 $\Delta(X)_n=X_n\oplus  X_{n+1}$, with differentials
 \[ d_n = \bbm 0&0\\1&0\ebm \colon  
     \Delta(X)_n = X_n\oplus  X_{n+1} \to X_{n-1}\oplus  X_n = \Delta(X)_{n-1}.
 \]
 The inclusions $X_n\to\Delta(X)_n$ give a morphism
 $i\in\Hom_{\Gr(\cat{B})}(X,\Delta(X))$ which is not a chain map.  The inclusions
 $X_{n+1}\to\Delta(X)_n$ give a chain map 
 \[
 j\in\Hom_{\Ch(\cat{B})}(\Sigma^{-1}X,\Delta(X)).
 \]
 We leave it to the reader to verify that the inclusion $i\colon X\to\Delta(X)$ is the unit map of an adjunction
\begin{equation}\label{eqn-Delta-adjunction}
 \Hom_{\Ch(\cat{B})}(\Delta(X),Y) \cong \Hom_{\Gr(\cat{B})}(X,Y).
\end{equation}
\end{Def}

The role of this construction is explained by the following result. 
\begin{Prop}\label{prop-im-Dl}
 Let $\CA$ be an abelian category, and let $\cat{P}$ be the full
 subcategory of projective objects.  Then an object $X\in\Ch(\CA)$ is
 contractible if and only if it is isomorphic to $\Delta(U)$ for some
 $U\in\Gr(\CA)$.
\end{Prop}
\begin{proof}
 First, we can define 
 \[ s_n = \bbm 0&1\\0&0\ebm \colon  
    \Delta(X)_n = X_n\oplus  X_{n+1} \to X_{n+1}\oplus  X_{n+2} = \Delta(X)_{n+1}.
 \]
 It is then easy to see that $ds+sd=1$ so $\Delta(X)$ is contractible.
 Conversely, suppose that $X\in\Ch(\CA)$ is contractible, and choose a
 contraction, i.e., a system of maps $s_n\colon X_n\to X_{n+1}$ with
 $ds+sd=1$.  Put $U_n=Z_{n-1}X=\ker(d\colon X_{n-1}\to X_{n-2})$, and
 define  
 \[ 
    p_n = \bbm s_{n-1} & 1 \ebm \colon  (\Delta U)_n = 
     Z_{n-1}X\oplus  Z_nX \to X_n.
 \]
 This satisfies 
 \[ pd-dp = \bbm s & 1\ebm \bbm 0&0\\1&0 \ebm - d\bbm s & 1\ebm 
     = \bbm 1 - ds & -d \ebm = \bbm sd & -d \ebm.
 \]
 We are considering this as a map defined on $Z_{n-1}X\oplus  Z_nX$ so
 the relevant differentials in the last expression are zero and we get
 $pd-dp=0$.  This shows that $p$ is a chain map.  Next, we define
 \[ i_n = \bbm d\\ ds\ebm\colon X_n\to Z_{n-1}X\oplus  Z_nX =(\Delta U)_n. \]
 The relation $ds+sd=1$ gives $pi=1\colon X\to X$.  Next, note that 
 \[ ds^2 = (1-sd)s = s(1-ds) = s^2d. \]
 Using this, we get 
 \[ ip = \bbm d\\ ds\ebm \bbm s & 1\ebm = 
     \bbm ds & ds^2 \\ d & ds \ebm = 
     \bbm 1-sd & s^2d \\ d & 1-sd \ebm.
 \]
 Again we are considering this as a map defined on
 $Z_{n-1}X\oplus  Z_nX$ so the relevant differentials are zero and we
 get that $ip$ is the identity.  Thus, $X$ is in the essential image of
 $\Delta$ as claimed. 
\end{proof}

From this we obtain two corollaries which we will need later in the paper.

\begin{Cor}\label{cor-contr-proj}
 Let $\CA$ be an abelian category, let $\cat{P}$ be the full
 subcategory of projective objects, and let $X$ be an object of $\Ch(\cat{P})$.
 Then the following are equivalent:
 \begin{itemize}
  \item[(a)] $X$ is contractible;
  \item[(b)] $X$ is isomorphic to $\Delta(U)$ for some $U\in\Gr(\cat{P})$;
  \item[(c)] $X$ is projective in $\Ch(\cat{A})$.
 \end{itemize}
\end{Cor}
\begin{proof}
 If~(a) holds then \cref{prop-im-Dl} tells us that $X\cong\Delta(U)$ for some
 $U\in\Gr(\cat{A})$.  This implies that each $U_n$ is a retract of $X_n$ and so lies in $\cat{P}$, so $U\in\cat{P}$, proving~(b).  If~(b) holds then the functor $\Hom_{\Ch(\cat{A})}(X,-)$ is equivalent to $\Hom_{\Gr(\cat{A})}(U,-)$, and this is clearly exact, proving~(c).  If~(c) holds then the natural epimorphism $\Delta(X)\to X$ must split, so $X$ is a retract of the contractible complex $\Delta(X)$ and so is itself contractible, proving~(a).
\end{proof}

\begin{Cor}\label{cor-ext-Ch}
 Let $\CA$ be an abelian category, and let $\cat{P}$ be the full
 subcategory of projective objects. Suppose that $X\in\Ch(\cat{P})$ and $Y\in\Ch(\cat{A})$. We then have
 \[ \Ext^t_{\Ch(\CA)}(X,Y) \cong \begin{cases}
      \Hom_{\Ch(\CA)}(X,Y) & \text{ if } t = 0 \\
      \Hom_{\K{\CA}}(\Sigma^{-t}X,Y) & \text{ if } t > 0.     
     \end{cases}
 \]
 In particular, if the differential in $X$ is zero and $t>0$ we have
 \[ \Ext^t_{\Ch(\CA)}(X,Y) \cong \prod_i \Hom_{\CA}(X_{i+t},H_i(Y)).
 \]  
\end{Cor}
\begin{proof}
 We have evident short exact sequences 
 $\Sigma^{-t-1}X\to \Delta\Sigma^{-t}X\to\Sigma^{-t}X$, in which the middle term is a projective object of $\Ch(\cat{A})$.  We can splice these to get a projective resolution of $X$, whose $t$-th term is $\Delta\Sigma^{-t}X$.  From this we get a complex calculating $\Ext_{\Ch(\CA)}^*(X,Y)$, whose $t$-th term is $\Hom_{\Gr(\CA)}(\Sigma^{-t}X,Y)$, which is the same as $\Hom(X,Y)_{-t}$.  One can check that the differential is just the usual one, and the claim follows.
\end{proof}

\subsection{Special abelian categories}
In this subsection we consider abelian categories in which every object is projective and show that in this case quasi-isomorphisms agree with homotopy equivalences. The prototypical example of an abelian category satisfying the above assumption is the category of vector spaces over a field, but this also happens in our context of global representations when the chosen family is a groupoid, see \cref{ex-unit-projective}.

\begin{Lem}\label{lem-semisimple-complex}
 Let $\CA$ be an abelian category  in which every
 object is projective.  If $X\in\Ch(\CA)$ then there is an isomorphism $X\cong X'\oplus  X''$
 where $X'$ has zero differential and $X''$ is contractible.  Thus,
 the inclusion $X'\to X$ and the projection $X\to X'$ are both chain
 homotopy equivalences, and they induce an isomorphism
 $X'\cong H_*(X)$.
\end{Lem}
\begin{proof}
 As the subobject $B_{k-1}X=\img(d_k\colon X_k\to X_{k-1})$ is projective,
 the epimorphism $d_k\colon X_k\to B_{k-1}X$ must split.  We can therefore
 choose $T_k\leq X_k$ such that $d_k$ induces an isomorphism
 $T_k\to B_{k-1}X$, and this implies that $X_k\cong Z_kX\oplus  T_k$.
 Similarly, the quotient map $Z_kX\to H_kX=(Z_kX)/(B_kX)$ must split,
 so we can choose $X'_k\leq Z_kX$ such that
 $Z_kX\cong X'_k\oplus  B_kX \cong X'_k\oplus  d(T_{k+1})$.  Thus, if we put
 $X''_k\cong T_k\oplus  B_kX \cong T_k\oplus  d(T_{k+1})$ then
 $X_k \cong X'_k\oplus  X''_k$.  It is clear by construction that $d=0$ on
 $X'$ and that $X''$ is a contractible subcomplex.
\end{proof}

\begin{Prop}\label{prop-semisimple-complex}
 Let $\CA$ be an abelian category in which every object is projective.
 Then every acyclic object of $\Ch(\CA)$ is contractible, and every
 quasi-isomorphism in $\Ch(\CA)$ is a homotopy equivalence. 
\end{Prop}
\begin{proof}
 Let $X\in\Ch(\CA)$ be acyclic, so $H_*(X)=0$. \cref{lem-semisimple-complex} gives a
 splitting $X\cong X'\oplus  X''$ where $X''$ is contractible and
 $X'\cong H_*(X)=0$, so $X$ is itself contractible as claimed.

 Now let $f\colon X\to Y$ be a quasi-isomorphism in $\Ch(\CA)$.  We can
 split $X$ as $X'\oplus  X''$ and $Y$ as $Y'\oplus  Y''$ as in \cref{lem-semisimple-complex}.  The
 composite 
 \[ g = (X'\xrightarrow{\text{inc}} X \xrightarrow{f} Y \xrightarrow{\text{proj}} Y') \]
 is then a quasi-isomorphism between complexes with zero differential,
 so it is an isomorphism in $\Ch(\CA)$.  We have also seen that the
 maps $\text{inc}$ and $\text{proj}$ are homotopy equivalences,
 so they become isomorphisms in $\K{\CA}$.  It follows that $f$ also
 becomes an isomorphism in $\K{\CA}$, so it is a homotopy
 equivalence as claimed.
\end{proof}

\begin{Cor}\label{cor-homology-equivalence-graded}
Let $\CA$ be an abelian category in which every object is projective. Then the homology functors define an equivalence
 \[
 H_* \colon \D{\CA}\xrightarrow{\sim } \Gr(\CA).
 \]
\end{Cor}
\begin{proof}
    We note that by \cref{lem-semisimple-complex} any object in the derived category $(X,d)$ is quasi-isomorphic to its homology $(H_*(X), 0)$. As any object in $\CA$ is projective, we see that all higher Ext-groups must vanish showing that the above functor is fully faithful. As the homology functor is clearly essentially surjective, it is an equivalence. 
\end{proof}

\begin{Cor}\label{cor-semisimple-symmon}
 Let $\CA$ be a closed symmetric monoidal abelian category in which
 every object is projective.  For any $U \in \Ch(\cat A)$, the functors 
 \[
 U \otimes-, \iHom(U,-)\colon \Ch(\CA) \to \Ch(\CA)
 \]
 preserve quasi-isomorphisms.
\end{Cor}
\begin{proof}
 The tensor product and internal hom functors are additive functors and so preserve homotopy equivalences.  Thus, the claim follows from \cref{prop-semisimple-complex}.
\end{proof}

\section{The abelian category \texorpdfstring{$\A{\cat{U}}$}{A(U)}}
\label{sec-AU}
In this section we introduce one of the main categories of interest
for this paper, the abelian category $\A{\cat{U}}$. We recall various
properties from~\cite{PolStrickland2022} and prove several
important features of the subcategory of projective objects
in $\A{\cat{U}}$.

\subsection{Definition and key properties}
We begin with the definition of $\A{\cat{U}}$.
\begin{Def}
    Let $\cat G$ be the category of finite groups and conjugacy classes of surjective group homomorphisms. Given a replete full subcategory $\cat U \subseteq \cat G$, we denote by 
        \[
            \A{\cat{U}} \coloneqq \Fun(\cat{U}^\mathrm{op}, \Mod{k})
        \]
    the additive category of functors $\cat{U}^\mathrm{op} \to \Mod{k}$ for a fixed field $k$ of characteristic zero (which will usually be omitted from the notation).
\end{Def}

\begin{Exa}
    Let $\{G\}\subseteq \cat G$ denote the replete full subcategory spanned by a single finite group $G$. In this case $\A{\{G\}}=\Mod{k[\Out(G)]}$. In this sense the abelian category $\A{\cat U}$ encodes the representation theory of the outer automorphism groups $\Out(G)$ for all $G\in\cat U$.
\end{Exa}

In our work below, we will often need to impose some closure properties on the replete full subcategory $\cat{U}\subseteq\cat{G}$. Let us recall some terminology. 
\begin{Def}
 Let $\cat{U}\subseteq\cat{G}$ be a replete full subcategory. Then $\cat U$ is:
 \begin{enumerate}
    \item \emph{closed downwards}: if $G \in \cat U$ and there exists a surjective group homomorphism $G \twoheadrightarrow H$, then $H \in \cat U$ too. Equivalently, if $G \in \cat{U}$ and $N\lhd G$ is a normal subgroup, then $G/N\in\cat{U}$;
    \item \emph{closed upwards}: if $H \in \cat U$ and there exists a surjective group homomorphism $G \twoheadrightarrow H$, then $G\in \cat U$, too. Equivalently, if $N\lhd G$ with $G/N\in\cat{U}$, then we have $G\in\cat{U}$;
    \item a \emph{groupoid}: if any surjective homomorphism between groups in $\cat{U}$ is an isomorphism; 
    \item \emph{widely closed}: if, whenever we have a pair of surjective homomorphisms $H_0\xleftarrow{}G\xrightarrow{}H_1$ with $G,H_0,H_1\in\cat{U}$, the image of the combined morphism $G\to H_0\times H_1$ is also in $\cat{U}$.  Equivalently: if $G$, $G/N_0$, and $G/N_1$ all lie in $\cat{U}$, then $G/(N_0\cap N_1)$ also lies in $\cat{U}$;  
    \item \emph{multiplicative}: if $G, H \in \cat U$ then $G \times H \in \cat U$;
    \item a \emph{global family}: if it is closed downwards and closed under subgroups.
 \end{enumerate}
 
\end{Def}

\begin{Rem}
 Wide closure is a mild condition, which follows from a wide range of other natural conditions that one might assume. For example, any of the conditions (a), (b), (c), and (f) imply that $\cat{U}$ is widely closed, see \cite[Remark 3.2]{PolStrickland2022}.
\end{Rem}

We now set the standing hypothesis that we will always work under for the rest of the paper, unless otherwise stated.
\begin{Hyp}\label{hyp-basic}
 For the rest of this paper, we assume that $\cat{U}$ is a replete full subcategory of $\cat G$ which is widely closed.
\end{Hyp}
With this assumption in place, we are ready to discuss in more detail the category of global representations.

\begin{Rec}\label{rec-pointwise}
   The category $\A{\cat U}$ inherits its basic properties from the category of $k$-modules through the collection of evaluation functors
   \begin{equation}\label{eq:evaluation}
       \psi_G\colon \A{\cat{U}} \to \Mod{k}, \quad X \mapsto X(G)
   \end{equation}
    for all $G\in\cat U$. For example:
    \begin{enumerate}
        \item any map admits a kernel and a cokernel and these are preserved by $\psi_G$;
        \item $\A{\cat{U}}$ is abelian and $\psi_G$ is exact;
        \item any small diagram admits a limit and a colimit and these are preserved by $\psi_G$;
        \item filtered colimits are exact.
    \end{enumerate}
\end{Rec}

In the next construction we record how these abelian categories are related to one another via change of families functors. 
\begin{Cons}\label{con-functors}
    Given a functor $f \colon \cat U \to \cat V$, we have adjunctions
    \[
    \begin{tikzcd}[column sep=large, row sep=large]
        \A{\cat{U}} \arrow[r,"f_!", yshift=2mm] \arrow[r, yshift=-2mm, "f_*"'] & \A{\cat{V}} \arrow[l, "f^*" description] 
    \end{tikzcd}
    \]
    with left adjoints displayed on the top. By definition $f^*(X)(H)=X(f(H))$, $f_!$ is given by the left Kan extension along $f$, and $f_*$ is given by the right Kan extension along $f$. We note that $f^*$ is always exact. If $f$ is fully faithful, then the theory of Kan extensions tells us that $f_!$ and $f_*$ are also fully faithful, and that if $f$ admits a left adjoint $q$, then $f_! \simeq q^*$, see \cite[Lemma 5.3]{PolStrickland2022} for more detail. 
\end{Cons}

\begin{Exa}
As an important special case of the previous construction, consider the inclusions $j_G\colon \{G\} \to \cat{U}$ for $G \in \cat{U}$. We obtain an adjoint triple 
\[
    \begin{tikzcd}[column sep=2.5cm, row sep=large]
        \A{\cat{U}} \arrow[r, "(j_G)^*" description]  & \A{\{G\}}. \arrow[l,"(j_G)_!"', yshift=2mm] \arrow[l, yshift=-2mm, "(j_G)_*"]
    \end{tikzcd}
    \]
    Since $j_G^*( X)=X(G)$ for all $X\in \A{\cat U}$ we will sometimes write $\mathrm{ev}_G$ for $(j_G)^*$. Note that composing this evaluation functor with the forgetful functor $\A{\{G\}} \to \Mod{k}$ agrees with the functor $\psi_G$ from \eqref{eq:evaluation}.
\end{Exa}

We now introduce some distinguished objects in the abelian category. 
\begin{Def}\label{defn-eGV}
    Let $\cat U\subseteq \cat G$ be a subcategory, $G \in \cat{U}$ and let $j_G\colon \{G\} \to \cat{U}$ denote the inclusion. For any $\Out(G)$-representation $V$, we define $e_{G,V} 
    \coloneqq (j_G)_!(V)$. For brevity, we set $e_G \coloneqq e_{G,k[\Out(G)]}$ and note that by construction
    \[
    \Hom_{\A{\cat U}}(e_G, X)\cong X(G)
    \]
    for all $X \in\A{\cat U}$.
\end{Def}

\begin{Rem}
    Unravelling the previous definition, one can give an explicit description of $e_{G,V}$ as 
    \[
    e_{G,V} \cong V \otimes_{k[\Out(G)]} k[\Hom_{\cat{U}}(-,G)]\in\A{\cat{U}}.
    \]
\end{Rem}

\begin{Rem}\label{rem-restriction-gen}
    A priori the objects $e_{G,V}$ from \cref{defn-eGV} depend on the ambient category $\cat U$ so we can indicate this by writing $e_{G,V}^{\cat U}$. However, given an inclusion $i\colon \cat{U} \to \cat{V}$, 
    it is immediate from the definition that we have 
    \[
    i_! e_{G,V}^{\cat U}\cong e_{G,V}^{\cat V} \quad \mathrm{and} \quad i^*e_{G,V}^{\cat V}\cong  e_{G,V}^{\cat U}
    \]
    for all $G \in \cat{U}$, so often there is no harm in dropping the superscript. 
\end{Rem}

Recall that an abelian category is Grothendieck if it has exact filtered colimits and a set of objects $\cat{X}$, called generators, such that any object $Y$ admits an epimorphism $\bigoplus X_i \to Y$ where $X_i \in \cat{X}$. As a formal consequence of \cref{defn-eGV} we obtain the following:
\begin{Prop}\label{prop-ebullet}
    Let $\cat U\subseteq \cat G$, $G \in \cat{U}$ and $V$ an $\Out(G)$-representation. The object $e_{G,V}$ is projective. Moreover, $\A{\cat{U}}$ is Grothendieck with a set of projective generators given by $\{e_G \mid G \in \cat{U}\}$.
\end{Prop}
\begin{proof}
    By Maschke's theorem, $V$ is projective in $\A{\{G\}}$. Since $e_{G,\bullet}$ is left adjoint to an exact functor it preserves projectives, so $e_{G,V}$ is projective in $\A{\cat{U}}$. Note that by the defining property of $e_G$, any $y\in Y(G)$ gives rise to a map 
    \[
    f_y \colon e_G \to Y
    \]
    such that $f_y(1_G)=y.$ Combining these morphisms, we get an epimorphism
    \[
    \bigoplus_{G \in \cat U}\bigoplus_{y \in Y(G)} e_{G} \to Y
    \] 
    showing that the set $\{e_G\}_{G \in \cat U}$ generates $\A{\cat{U}}$. 
\end{proof}
For objects in a Grothendieck abelian category, there is a notion of finite generation which we now recall.
\begin{Def}
 An object $X \in \A{\cat U}$ is \emph{finitely generated} if it is a quotient of a finite sum  of generators, that is if there exists an epimorphism $\bigoplus_{i=1}^ne_{G_i} \twoheadrightarrow X$.  
\end{Def}
It is easy to see that finitely generated objects in a Grothendieck abelian category are closed under quotients and extensions. However, they are not generally closed under subobjects. To address this issue, one often considers locally noetherian Grothendieck abelian categories. A Grothendieck abelian category $\cat A$ is said to be \emph{locally noetherian} if it has a generating set $\cat X$ of noetherian objects; this means that for all generators $X \in \cat X$, any subobject of $X$ is finitely generated. It then follows that subobjects of finitely generated objects are again finitely generated, and so the subcategory $\cat{A}^{\mathrm{fg}}\subseteq \cat A$ of finitely generated objects is abelian. For this reason it is important to have a good supply of examples and counterexamples of families $\cat U$ for which $\A{\cat U}$ is locally noetherian.

\begin{Prop}[{\cite[Theorem A]{PolStrickland2022}}]
    Let $\cat U$ be a subcategory of $\cat G$ and let $p$ be a prime number. 
    \begin{enumerate}
        \item If $\cat U$ is a multiplicative global family of finite abelian $p$-groups, then $\A{\cat U}$ is locally noetherian. 
        \item If $\cat U$ is the subcategory of cyclic $p$-groups, then $\A{\cat U}$ is locally noetherian. 
        \item If $\cat U$ contains the trivial group and infinitely many cyclic groups of prime order, then $\A{\cat U}$ is not locally noetherian. 
    \end{enumerate}
\end{Prop}
We next discuss the symmetric monoidal structure of global representations.

\begin{Rec}
Recall that $\A{\cat{U}}$ has a standard closed symmetric monoidal
structure.  The tensor product is given pointwise by
$(X\otimes Y)(G)=X(G)\otimes Y(G)$, and the tensor unit is the
constant functor $\unit(G)=k$ for all $G$ in $\cat{U}$. Since the tensor product is defined
pointwise, it is clear that it is exact, and hence any object in
$\A{\cat{U}}$ is flat. If $\cat{U}$ contains the trivial group $1$, then $\unit \cong e_1$. The ordinary morphism set
$\Hom_{\A{\cat{U}}}(X,Y)\in\Mod{k}$ and the internal mapping object
$\iHom(X,Y)\in\A{\cat U}$ are related by 
\[
    \iHom(X,Y)(G) \cong \Hom_{\A{\cat{U}}}(e_G\otimes X,Y).
\]
In particular, $\Hom_{\A{\cat{U}}}(X,Y)\cong \iHom(X,Y)(1)$ when $1 \in \cat{U}$. This structure
propagates to $\Gr(\A{\cat{U}})$ and $\Ch(\A{\cat{U}})$ as discussed in \cref{rec-internal_homs}. For example, there are internal mapping objects $\iHom(X,Y)\in\Gr(\A{\cat{U}})$ and $\Hom(X,Y)\in\Gr(\Vect)$ related by
\begin{align*}
 \Hom(X,Y)_n &\cong \Hom_{\Gr(\A{\cat{U}})}(\Sigma^nX,Y) \cong \iHom(X,Y)_n(1); \\
 \iHom(X,Y)_n(G) &\cong \Hom_{\Gr(\A{\cat{U}})}(e_G\otimes\Sigma^nX,Y) \cong  \Hom(e_G\otimes X,Y)_n; \\
 \Hom_{\Gr(\A{\cat{U}})}(X,Y) &\cong \Hom(X,Y)_0 \cong \iHom(X,Y)_0(1),
\end{align*}
where the statements involving the trivial group $1$ are conditional on $1$ being contained in $\cat{U}$. If $X,Y$ have differentials, then the internal mapping objects are naturally chain complexes and the same relations as above hold. 
\end{Rec}
As an immediate consequence of the pointwise symmetric monoidal structure we obtain the following:

\begin{Lem}
    Let $f\colon \cat{U} \to \cat{V}$ be a functor. The functor $f^*\colon \A{\cat V} \to \A{\cat U}$ is strong symmetric monoidal and hence $f_!$ (resp., $f_*$) admits a canonical oplax (resp., lax) symmetric monoidal structure.
\end{Lem}

\begin{Lem}\label{lem-flat}
 For $X,Y\in\Ch(\A{\cat{U}})$ we have a natural isomorphism
 \[H_*(X)\otimes H_*(Y)\to H_*(X\otimes Y)\] in $\Gr(\A{\cat{U}})$.  Thus, if $f\colon X\to X'$
 and $g\colon Y\to Y'$ are quasi-isomorphisms, then so is $f\otimes g$.
\end{Lem}
\begin{proof}
 The evaluation functors $\psi_G\colon \A{\cat{U}} \to \Mod{k}$ from \eqref{eq:evaluation} are exact, strong monoidal, and jointly conservative. Therefore it suffices to show that the natural pairing
 $H_*(X(G))\otimes H_*(Y(G))\to H_*(X(G)\otimes Y(G))$ is an isomorphism for
 all $G$.  This is just a statement about the category of vector spaces, which
is well known and easily follows from \cref{lem-semisimple-complex}. 
\end{proof}

\subsection{Projective objects}
In this subsection we study the full subcategories 
\[
\mathsf{P}_{\fg}(\cat U)\subseteq \P{\cat{U}}\subseteq \A{\cat U}
\]
spanned by the finitely generated projective objects and the projective objects.

\begin{Rem}\label{rem-projective-objects}
 For any $\Out(G)$-representation $V$, we saw in \cref{prop-ebullet} that
 $e_{G,V} \in \A{\cat{U}}$ is projective. In fact, any projective
 object in $\A{\cat{U}}$ is a direct sum of objects of the form
 $e_{G,V}$, see \cite[Corollary 8.5]{PolStrickland2022}. 
\end{Rem}

In almost all of the cases that we want to consider, either~(a) $\cat{U}$ contains the trivial group, or~(b) $\cat{U}$ is a groupoid with only finitely many isomorphism classes. In both of those cases one can check that the tensor unit $\mathbb{1}$ is a finitely generated projective object, but this need not always be the case. Therefore, we next identify all subcategories for which the tensor unit is a finitely generated projective.

\begin{Def}\label{def-minimal}
 Let $\cat{U}$ be a widely closed subcategory of $\cat{G}$, and let $G$ be a group in $\cat{U}$. We say that $G$ is \emph{minimal} in $\cat{U}$ if every morphism $G\to G'$ in $\cat{U}$ is an isomorphism.
\end{Def}

\begin{Prop}\label{prop-minimal}
 The tensor unit $\unit\in\A{\cat{U}}$ is a projective object if and only if for each $G\in\cat{U}$, there is a unique (up to isomorphism) normal subgroup $N\lhd G$ such that $G/N$ is a minimal object of $\cat{U}$. It is furthermore finitely generated if and only if there are only finitely many isomorphism classes of minimal objects in $\cat{U}$.
\end{Prop}
\begin{proof}
    First, for any $M\in\cat{U}$ we put $c_M=e_{M,k}$, so 
    \begin{align*} 
    c_M(T) &\cong k[\Hom_{\cat{U}}(T,M)]\otimes_{k[\Out(M)]}k \\ & \cong 
        k[\Hom_{\cat{U}}(T,M)/\Out(M)] \\
        & \cong   k[\mathrm{Epi}(T,M)/\mathrm{Aut}(M)].
    \end{align*}
    Let $s_M(T)$ be the set of normal subgroups $N\lhd T$ such that $T/N\cong M$; then the map $[\alpha]\mapsto\ker(\alpha)$ gives a bijection $\mathrm{Epi}(T,M)/\Aut(M)\to s_M(T)$.  Thus, we can identify $c_M(T)$ with $k[s_M(T)]$.
    
    Now suppose that the condition on the minimal objects hold for each group in $\cat U$.  Choose a list $\{M_i\}_{i\in I}$ containing one representative of each isomorphism class of minimal objects, and put $U=\bigoplus_{i}c_{M_i}$.  There is an evident natural map $\theta_G\colon U(G)\to k$ sending each element of $s_{M_i}(G)$ to $1$ for all $i$. For each $G\in\cat{U}$ there is a unique normal subgroup $N\lhd G$ such that $G/N$ is minimal in $\cat{U}$, so there is a unique $i$ such that $G/N\cong M_i$, and we find that $c_{M_i}(G)\cong k$ and $c_{M_j}(G)=0$ for $j\neq i$, so $\theta_G$ is an isomorphism.  Thus $\mathbb{1}$ is isomorphic to $U$ and so is projective. It is furthermore finitely generated if $I$ is finite. 

    Conversely, suppose that $\mathbb{1}$ is projective.  Using \cref{rem-projective-objects} we see that $\mathbb{1}\cong\bigoplus_{i\in I}e_{M_i,V_i}$ for some family of objects $M_i\in\cat{U}$ and $\mathrm{Out}(M_i)$-representations $V_i$.  If $M_i\cong M_j$ then we can merge $V_i$ and $V_j$, and if $V_i=0$ then we can drop $M_i$.  Thus, we can assume that $V_i\neq 0$ for all $i$, and that $M_i\not\cong M_j$ for all $i\neq j$.  Now the module $V_i=e_{M_i}(M_i)$ is a non-zero summand in $\mathbb{1}(M_i)=k$, so we must have $V_i=k$ for all $i$, so $\unit \cong \bigoplus_ic_{M_i}$, so for all $G\in\cat{U}$ we have $\sum_i|s_{M_i}(G)|=1$.  This means that for all $G$ there is a unique pair $(N,i)$ such that $N\lhd G$ and $G/N\cong M_i$.  If $G$ is minimal then the morphism $G\to G/N\cong M_i$ must be an isomorphism, so $G$ is isomorphic to one of the groups $M_i$; this in particular proves that if $\unit$ is finitely generated projective then there are finitely many isomorphism classes of minimal objects.  More generally, suppose that $N'$ is any normal subgroup of $G$ such that $G/N'$ is a minimal object of $\cat{U}$.  By applying the previous remark to $G/N'$, we see that $G/N'\cong M_{i'}$ for some $i'$, but the pair $(N,i)$ was unique, so $N'\cong N$.  This proves the final claim. 
\end{proof}

\begin{Exas}\leavevmode\label{ex-unit-projective}
 \begin{itemize}
     \item[(a)] If $1\in\cat{U}$ then clearly it is the only minimal object in $\cat{U}$ so the proposition tells us that $\mathbb{1}$ is a finitely generated projective object, but this is also obvious anyway because $\mathbb{1}=e_1$ in this case.
     \item[(b)] If $\cat{U}$ is a groupoid then~\cite[Proposition 7.3]{PolStrickland2022} shows that all objects of $\A{\cat{U}}$ are projective, and it is clear that all objects of $\cat{U}$ are minimal.  If $\cat{U}$ has only finitely many isomorphism classes then $\mathbb{1}$ will also be finitely generated.  However, this will not hold if $\cat{U}$ has infinitely many isomorphism classes, for example if $\cat{U}$ is the category of finite simple groups.  Alternatively, we can choose any infinite set $P$ of positive integers such that no element of $P$ divides any other element of $P$, then let $\cat{U}$ be the full subcategory of $\cat{G}$ consisting of groups whose order lies in $P$.  Then $\CU$ is again a groupoid with infinitely many isomorphism classes.
     \item[(c)] Fix a prime $p$ and an integer $m>0$, and let $\cat{U}$ be the category of cyclic groups of order $p^k$ with $k\geq m$.  Then $C_{p^m}$ is (up to isomorphism) the unique minimal object in $\cat{U}$, and for any $G\in\cat{U}$ the subgroup $N=\{g^{p^m}\mid g\in G\}$ is the unique normal subgroup such that $G/N\cong C_{p^m}$.  It follows that $\mathbb{1}$ is isomorphic to $c_{C_{p^m}}$, and so is a finitely generated projective object. 
 \end{itemize}
\end{Exas}

\begin{Prop}\label{prop-proj-closed-under-tensor}
 The category $\P{\cat{U}}$ of projective objects is closed under direct sums, products and tensor products. If $\unit \in \A{\cat U}$ is projective then $\P{\cat U}$ is a symmetric monoidal subcategory of $\A{\cat U}$. Finally, if $\cat{U}$ is a multiplicative global family, then  $\P{\cat{U}}$ is closed under internal homs. Similar claims also hold for $\mathsf{P}_{\fg}(\cat U) \subseteq \mathsf{A}_{\fg}(\cat U)$.
\end{Prop}
\begin{proof}
 See~\cite[Propositions 8.6 and 8.7]{PolStrickland2022}. Note that similar proofs also work in the finitely generated setting.
\end{proof}

\section{DG-projective vs projective objects}

In this section we show that any complex of projective objects in $\A{\cat U}$ is DG-projective. This is a special feature of the abelian categories $\A{\cat{U}}$. As a consequence we deduce that any quasi-isomorphism between complexes of projectives is a homotopy equivalence. We will use this result in the next section to give a more concrete description of the derived category of $\A{\cat U}$. Finally, we give a counterexample for the dual statement: not every complex of injective objects is DG-injective. Recall that \cref{hyp-basic} is in place throughout.

We begin by recalling the key definitions for this section. For more details, we refer the reader to \cite{Spaltenstein} and \cite{AvramovFoxbyHalperin03}.

\begin{Def}\label{def-k-proj-dg-proj}
    Consider $X\in \Ch(\A{\cat U})$. We say that $X$ is:
    \begin{enumerate}
        \item $K$-\emph{projective} if $\Hom_{\K{\A{\cat U}}}(X, N)=0$ for all $N \in \Ch(\A{\cat U})$ acyclic.
        \item \emph{DG}-\emph{projective} if $X_n$ is projective for all $n\in\bbZ$ and $X$ is $K$-projective. 
    \end{enumerate}   
    Dually we say that $X$ is:
    \begin{enumerate}
        \item $K$-\emph{injective} if $\Hom_{\K{\A{\cat U}}}(N, X)=0$ for all $N \in \Ch(\A{\cat U})$ acyclic.
        \item \emph{DG}-\emph{injective} if $X_n$ is injective for all $n\in\bbZ$ and $X$ is $K$-injective.
    \end{enumerate}
\end{Def}

\begin{Rem}
   Recall for example from \cite[Theorem 10.2.15]{Yekutieli} that $K$-projective objects and $K$-projective resolutions can be used to calculate left derived functors. This highlights the importance of these definitions in homological algebra. Recall also that DG-projective objects can be described as the cofibrant objects in the projective model structure on the category of chain complexes (when it exists), see~\cite[Lemmas 2.3.6 and 2.3.8]{Hovey} or~\cite[Theorem 9.10]{sixmodels} for some instances. We show that the projective model structure on $\Ch(\A{\cat{U}})$ exists in \cref{thm-proj-model-structure}.
\end{Rem}

\begin{Rem}
    In a general abelian category, one can show that any bounded below complex of projective objects is DG-projective. However, in general not all complexes of projective objects are DG-projective. A counterexample is given in \cite[Remark 2.3.7]{Hovey}.
\end{Rem}

Before we can prove the claim about DG-projective complexes we need some preparation. 

\begin{Def}\label{defn-pure}
 Let $s$ be a positive integer, and let $\cat{U}_{=s}$ be the full
 subcategory of $\cat{U}$ consisting of groups of order $s$.  This is a
 groupoid (and therefore is widely closed), so all objects of
 $\A{\cat{U}_{=s}}$ are projective by~\cite[Proposition 6.3]{PolStrickland2022}.  Let
 $i\colon \cat{U}_{=s}\to\cat{U}$ be the inclusion which gives rise to an adjunction
 \[
 i_!\colon \A{\cat{U}_{=s}} \rightleftarrows \A{\cat{U}}\noloc i^*
 \] in which the left adjoint $i_!$ is fully faithful. We say that an object 
 $Q\in\A{\cat{U}}$ is \emph{$s$-pure} if it satisfies any of the following
 conditions, which are equivalent by~\cite[Proposition 8.3]{PolStrickland2022}:
 \begin{itemize}
  \item[(a)] $Q\cong i_!(P)$ for some $P\in\A{\cat{U}_{=s}}$;
  \item[(b)] $Q$ is a retract of $i_!(P)$ for some $P\in\A{\cat{U}_{=s}}$;
  \item[(c)] $Q$ is a retract of some direct sum of objects $e_G$ with
   $G\in\cat{U}_{=s}$;
  \item[(d)] the counit map $i_!i^*(Q)\to Q$ is
   an isomorphism.
 \end{itemize}
  We note that any $s$-pure object is necessarily projective by (c).
  We will also say that an object $X$ in $\Gr(\A{\cat{U}})$
 or $\Ch(\A{\cat{U}})$ is $s$-pure if $X_n$ is $s$-pure for all $n$.
\end{Def}

\begin{Prop}\label{prop-pure-split}
 Fix a positive integer $s$.  Let $X$ be an $s$-pure chain complex in
 $\Ch(\A{\cat{U}})$. Then there is an isomorphism $X \cong X' \oplus X''$ in $\Ch(\A{\cat{U}})$ where $X'$ has zero differential and $X''$ is contractible.
\end{Prop}
\begin{proof}
 The category of $s$-pure objects in $\A{\cat{U}}$ is equivalent to the
 category $\A{\cat{U}_{=s}}$ in which all objects are projective, so the
 claim follows from \cref{lem-semisimple-complex}.
\end{proof}
We next recall the following construction from \cite[Construction 5.11]{PolStrickland2022} which will play an important role in this paper. 
\begin{Cons}\label{cons-filtration}
  Let $s \geq 0$ and let $\cat{U}_{\leq s}$ be the full subcategory of $\cat{U}$ consisting of groups of order $\leq s$. Consider the induced adjunction
  \[
 i_!^{\leq s}\colon \A{\cat{U}_{\leq s}} \rightleftarrows \A{\cat{U}} \noloc i^*_{\leq s}.
 \]
  For $X \in \A{\cat{U}}$, we define $L_{\leq s}X$ to be the image of
  the counit map $i_!^{\leq s}i^*_{\leq s}X \to X$. Equivalently,
  $L_{\leq s}X$ is the smallest subobject of $X$ that contains $X(G)$
  whenever $|G|\leq s$. When convenient we will also write $L_{<s}X$ for 
  $L_{\leq(s-1)}X$.  This gives an ascending filtration  
  \[
  0= L_{\leq 0}X \subseteq L_{\leq 1}X \subseteq \ldots \subseteq L_{\leq s}X \subseteq L_{\leq s+1}X \subseteq \ldots  
  \]
  with subquotients denoted by $L_sX=L_{\leq s}X/L_{<s}X$. Given a map $f\colon X \to Y$ of objects in $\A{\cat{U}}$, one obtains maps $L_{\leq s}f \colon L_{\leq s}X \to L_{\leq s}Y$ and $L_s f \colon L_s X \to L_s Y$, so this filtration is natural in $X$. We will extend this filtration to the category of chain complexes by applying these functors at each degree of the chain complex. 
\end{Cons}
\begin{Exa}
    Given $G \in \cat U$ and an $\Out(G)$-representation $V$, one sees that $L_{\leq s}e_{G,V}$ is $e_{G,V}$ if $|G|\leq s$ and zero otherwise, compare \cref{rem-restriction-gen}. In particular $L_s e_{G,V}\cong e_{G,V}$ precisely if $s=|G|$, and is zero otherwise.  
\end{Exa}

\begin{Rem}\label{rem-filt}
 It is proved in~\cite[Proposition 8.4]{PolStrickland2022} that if $X$ is
 projective then $L_sX$ is $s$-pure and that $X$ is unnaturally isomorphic to $\bigoplus_sL_sX$.
 We now put $L_{\oplus }X=\bigoplus_sL_{\leq s}X$, and we let
 $\Sigma\colon L_{\oplus }X\to L_{\oplus }X$ be the direct sum of the inclusions
 $L_{\leq s}X\to L_{\leq(s+1)}X$.  It is easy to check that $1-\Sigma$ is
 injective with cokernel naturally isomorphic to $X$.  Thus, we have a
 short exact sequence 
 \[0 \to L_{\oplus }X \xrightarrow{1-\Sigma} L_{\oplus }X \xrightarrow{} X \to 0, \]
 which will be chainwise split if $X$ is projective. 
\end{Rem}
As a first application of this filtration we prove the following result which we will need later in the paper. The next result shows that in particular any complex of projective objects in $\A{\cat{U}}$ is $K$-projective and hence DG-projective. 

\begin{Thm}\label{thm-cofibrant}
 If $X\in\Ch(\P{\cat U})$ and $Y\in \Ch(\A{\cat U})$ is acyclic, then 
    \[ 
        \Hom_{\K{\A{\cat{U}}}}(\Sigma^t X,Y) = 0 
    \]
 for all $t\in \bbZ$.
\end{Thm}
\begin{proof}
 Let $\cat{S}$ be the subcategory of objects $X'\in\Ch(\P{\cat{U}})$ that have
 the property $\Hom_{\K{\A{\cat{U}}}}(\Sigma^t X',Y) = 0$ for all $t\in \bbZ$; we must show that any $X \in \Ch(\P{\cat U})$ is in $\cat{S}$.
 It is easy to see that $\cat{S}$ is closed under arbitrary direct sums, retracts, suspensions, and desuspensions. All contractible complexes lie in $\cat{S}$, and \cref{cor-ext-Ch} tells us that complexes with zero differential lie in $\cat{S}$.   Using \cref{prop-pure-split}
 we now see that every $s$-pure complex lies in $\cat{S}$.  

 Now consider an arbitrary object $X\in\Ch(\P{\cat{U}})$.  The complexes
 $L_sX$ are $s$-pure by \cref{rem-filt} and so lie in $\cat{S}$.  The short exact sequences
 \[0 \to L_{<s}X\to L_{\leq s}X\to L_{s}X \to 0\] give rise to distinguished triangles in $\K{\A{\cat{U}}}$, using which we can prove by induction
 that $L_{\leq s}X\in\cat{S}$ for all $s$, hence $L_{\oplus }X$ also lies in
 $\cat{S}$.   The short exact sequence 
 \[0 \to L_{\oplus }X \xrightarrow{1-\Sigma} L_{\oplus }X \xrightarrow{} X \to 0 \]
 also gives a distinguished triangle, so we deduce that $X\in\cat{S}$ as required. 
\end{proof}

\begin{Cor}\label{cor-no-ext}
 Let $X\in \Ch(\P{\cat U})$. Then
 for all $t>0$ and all acyclic complexes $Y\in\Ch(\A{\cat{U}})$ we have
 $\Ext^t_{\Ch(\A{\cat{U}})}(X,Y)=0$.
\end{Cor}
\begin{proof}
 Combine \cref{thm-cofibrant} with \cref{cor-ext-Ch}.
\end{proof}

We record some formal consequences. 

\begin{Cor}\label{cor-acyc-proj}
 If $X\in \Ch(\P{\cat U})$ is an acyclic complex, then $X$ is contractible, and so is projective in $\Ch(\A{\cat{U}})$.
\end{Cor}

\begin{proof}
    The identity map on $X$ must be chain homotopic to zero by \cref{thm-cofibrant}, and hence $X$ is contractible. Therefore it is a projective object in  $\Ch(\A{\cat{U}})$ by \cref{cor-contr-proj}. 
\end{proof}

\begin{Cor}\label{cor-cofibrant}
 Suppose that $X\in\Ch(\P{\cat{U}})$ and that $Y\in\Ch(\A{\cat{U}})$ is
 acyclic.  Then $\iHom(X,Y)$ is also acyclic. 
\end{Cor}
\begin{proof}
 We need to show that $H_n(\iHom(X,Y)(G))=0$ for all $n\in\Z$ and
 $G\in\cat{U}$.  This group can be identified with
 $\Hom_{\K{\A{\cat{U}}}}(X',Y)$ where $X'=\Sigma^n e_G\otimes X$.  By \cref{prop-proj-closed-under-tensor} we deduce that $X'\in\Ch(\P{\cat{U}})$, so
 $\Hom_{\K{\A{\cat{U}}}}(X',Y)=0$ by \cref{thm-cofibrant}.
\end{proof}

\begin{Cor}\label{cor-cofibrant-alt}
 Suppose that $X\in\Ch(\P{\cat U})$ and that $f\colon Y\to Z$ is a
 quasi-isomorphism in $\Ch(\A{\cat{U}})$.  Then the induced map 
 \[ 
 f_* \colon  \Hom_{\K{\A{\cat{U}}}}(X,Y) \to \Hom_{\K{\A{\cat{U}}}}(X,Z) 
 \]
 is an isomorphism.  In particular, any quasi-isomorphism $f \colon Y \to Z$ in $\Ch(\P{\cat U})$ is a homotopy equivalence.
\end{Cor}
\begin{proof}
 The mapping cone of $f$ is acyclic, so the first claim follows from the long exact sequence associated to the triangle $Y \xrightarrow{f} Z \to C(f)$, together with \cref{thm-cofibrant}. Let us now discuss the last claim. By setting $X=Z$ in the previous claim we find that there exists a map $g \colon Z \to Y$ such that $fg \simeq \mathrm{id}_Z$. By the $2$-out-of-$3$ property we deduce that $g$ is a quasi-isomorphism, and so by applying the same argument to $g$ we find that there exists $h \colon Y \to Z$ such that $gh \simeq \mathrm{id}_Y$. Finally, $f\simeq f(gh) \simeq (fg)h\simeq h$, concluding the proof.
\end{proof}

In contrast to the case with projectives, it is not true that every
 complex of injectives of $\A{\cat{U}}$ is DG-injective, and it is not true that
 every acyclic complex of injectives is contractible. We end this section with an explicit counterexample.
\begin{Exa}
  For any set $T$ we have a simplicial set $ET$ with
 \[
 (ET)_n=\Map([n],T)\cong T^{n+1}.
 \]
 By the Dold--Kan correspondence, this gives rise to a chain complex $k ET$, where $(k ET)_n$ is the vector space over $k$ with basis $(ET)_n$.  The homology of
 $k ET$ is just a copy of $k$ concentrated in dimension $0$, except
 in the case where $T=\emptyset$, in which case $H_*(k ET)=0$.

 Let $\cat{U}$ be the family of finite $p$-groups and put $X(G)=k E\Hom_{\cat{U}}(G,C_p)$. We claim that the complex $X$ can be adjusted to provide
 a counterexample. The complex $X$ has $X_n\cong e_{C_p}^{\otimes(n+1)}$ for $n\geq 0$, so each $X_n$ is projective by \cref{prop-proj-closed-under-tensor}, and therefore also injective 
 by~\cite[Proposition 15.1]{PolStrickland2022} as $\cat{U}$ is a multiplicative global family.  We have $H_*(X)(1)=0$ and  
 $H_*(X)(G)\cong H_0(X)(G)\cong k$ for $G\neq 1$ as explained above.  
 We now construct a new complex $Z$ with $Z_n=X_n$ for $n\geq 0$ and $Z_{-1}=\unit$ and 
 \[ Z_{-2}(G) \cong \begin{cases}
     k & \text{ if } G\cong 1 \\
     0 & \text{ otherwise}.
 \end{cases}\]
 There is an evident way to define the differentials so that $Z$ is 
 acyclic.  Note that the object $Z_{-1}=e_1$ is also projective and 
 therefore injective.  Moreover, there is a natural isomorphism 
 $\Hom_{\A{\cat{U}}}(T,Z_{-2})\cong\Hom_k(T(1),k)$ which shows that $Z_{-2}$ 
 is also injective (but not projective).
 
 We now see that $Z$ is an acyclic complex of injectives, so if 
 it were DG-injective, there would be a contraction, consisting of
 morphisms $s_i\colon Z_i\to Z_{i+1}$.  As $Z_{-3}=0$, we would have
 $d_{-1}\circ s_{-2}=1\colon Z_{-2}\to Z_{-2}$.  However, it is easy to see
 that  
 \[ \Hom_{\A{\cat{U}}}(Z_{-2},Z_{-1})=0, \]
 so $s_{-2}=0$, which gives a contradiction.
\end{Exa}

\section{A model for the derived category}

In this section we show that the homotopy category of chain complexes of projectives is a
model for the derived category of $\A{\cat{U}}$, see \cref{thm-derived-category-equal-K-proj}. We remind the reader that throughout the paper we are working under \cref{hyp-basic}.

We start by setting the notation for the derived category.

\begin{Def}
    Let $\cat U \subseteq \cat G$ be a subcategory. We write 
    \[
    \D{\cat U}\coloneqq \D{\A{\cat U}}=\K{\A{\cat U}}[\mathrm{q.iso}^{-1}]
    \]
    for the derived category of $\A{\cat U}$, obtained from the homotopy category by inverting the quasi-isomorphisms.
\end{Def}

\begin{Rem}
    As tensor products of quasi-isomorphisms are
 quasi-isomorphisms (see \cref{lem-flat}), the derived category $\D{\cat{U}}$ inherits a
 symmetric monoidal structure which is compatible with that of $\A{\cat U}$ in the sense that there is a commutative diagram 
 \[
\begin{tikzcd}
\A{\cat U} \times \A{\cat U} \arrow[r,"-\otimes -"] \arrow[d, hook] & \A{\cat U} \arrow[d, hook] \\
\sfD(\cat U) \times \sfD(\cat U) \arrow[r, "-\otimes-"']& \sfD(\cat U),
\end{tikzcd}
\]
where the vertical arrows are the inclusions as complexes in degree zero. 
\end{Rem}

The main task of this section is to construct an exact functor $\D{\cat{U}} \to \K{\P{\cat{U}}}$. This will rely on the following construction.

\begin{Rec}\label{rec-projectiveresolution}
Recall that~\cite[Construction 5.6]{PolStrickland2022} gives a projective
resolution $P_*(X)$ for any $X\in\A{\cat{U}}$, which extends to a $k$-linear functor in $X$. 
For the benefit of the reader we recall the construction. 
Write $\cat{U}^\times$ for the subcategory of $\cat{U}$ with all
 objects but only the invertible morphisms, and
 $l\colon \cat{U}^\times\to\cat{U}$ for the inclusion. This gives an
 adjunction
 \[l_!\colon \A{\cat{U}^\times} \rightleftarrows l^*\noloc \A{\cat{U}}\]
 and so we have a counit map $\epsilon_X \colon l_!l^* X \to X$. We set  
 \begin{align*}
 P_0(X)&=l_!l^*(X) \\
 P_1(X)&=l_!l^*(\ker(P_0(X) \to X))\\
 P_{i+2}(X)&=l_!l^*(\ker(P_{i+1}(X) \to P_i(X))) \;\; \forall i\geq 0.
 \end{align*}
 We write $\delta$ for the map $P_{i+1}(X)\to P_i(X)$ appearing in the
 above construction, which is the differential in the complex
 $P_*(X)$.  One can check that if $\dim_k(X(G))<\infty$ for all $G$,
 then the same is true of $P_nX$. Moreover, \cite[Remark
 5.7]{PolStrickland2022} tells us that if $X(G)=0$ whenever $|G|<d$,
 then $(P_nX)(G)=0$ whenever $|G|<d+n$. Thus, after evaluation at any $G$ this becomes a bounded chain complex.
\end{Rec}

We now check one extra property of the functor $P_*$. 

\begin{Lem}\label{lem-P-exact} 
    The functor $P_*\colon \A{\cat{U}}\to\Ch(\A{\cat{U}})$ preserves finite limits and all colimits; in particular, it is exact.
\end{Lem}
\begin{proof}
    Recall that $P_0(X)=l_!l^*(X)$, where $l\colon\cat{U}^\times\to\cat{U}$ is the inclusion.  It is clear that $l^*$ preserves all limits and colimits, and $l_!$ is a left adjoint so it preserves all colimits. It follows that $P_0$ also preserves all colimits.

 Consider a short exact sequence $X\xrightarrow{i}Y\xrightarrow{p}Z$
 in $\A{\cat{U}}$.  Then $l^*(X)\to l^*(Y)\to l^*(Z)$ is a short exact
 sequence in $\A{\cat{U}^\times}$, but $\cat{U}^\times$ is a
 groupoid so all short exact sequences in $\A{\cat{U}^\times}$ split 
 by \cite[Proposition 6.3]{PolStrickland2022}.  It follows that we
 still have a short exact sequence after applying $l_!$, so $P_0$ is
 exact.  We now have a diagram 
 \begin{center}
  \begin{tikzcd}
   \ker(\epsilon_X) \arrow[r] \arrow[d,rightarrowtail] & 
   \ker(\epsilon_Y) \arrow[r] \arrow[d,rightarrowtail] & 
   \ker(\epsilon_Z)           \arrow[d,rightarrowtail] \\
   P_0(X) \arrow[r,rightarrowtail]    \arrow[d,"\epsilon_X"',twoheadrightarrow] &
   P_0(Y) \arrow[r,twoheadrightarrow] \arrow[d,"\epsilon_Y",twoheadrightarrow] &
   P_0(Z)                             \arrow[d,"\epsilon_Z",twoheadrightarrow] \\
   X \arrow[r,rightarrowtail] &
   Y \arrow[r,twoheadrightarrow] &
   Z.
  \end{tikzcd}
 \end{center}
 We have seen that the middle row and the bottom row are short exact,
 so the snake lemma implies that the top row is also short
 exact.  The sequence $P_1(X)\to P_1(Y)\to P_1(Z)$ is obtained by
 applying $l_!l^*$ to the top row, so it is again short exact.  An
 evident inductive extension of this shows that the sequence
 $P_*(X)\to P_*(Y)\to P_*(Z)$ is a short exact sequence of chain
 complexes.  It follows in a standard way that $P_*$ preserves all
 finite limits and colimits: as $l^*$ and $l_!$
 preserve all coproducts, it follows easily that $P_*$ has the same
 property, and as arbitrary colimits can be built from coproducts and
 cokernels, they are also preserved by $P_*$.
\end{proof}
We next verify that the above functor extends to chain complexes. 
\begin{Cons}\label{cons-P}
 For any object $X\in\Ch(\A{\cat{U}})$, there are commutative
 squares 
 \begin{center}
  \begin{tikzcd}[row sep=1.1cm, column sep=1cm]
   P_i(X_j) \arrow[d,"\delta"'] \arrow[r,"P_i(d_j)"] & 
   P_i(X_{j-1}) \arrow[d,"\delta"] \\
   P_{i-1}(X_j) \arrow[r,"P_{i-1}(d_j)"'] &
   P_{i-1}(X_{j-1}).
  \end{tikzcd}
 \end{center}
 We can totalise this and consider the complex $P(X)$ with
 $P(X)_m=\bigoplus_{m=i+j}P_i(X_j)$, with the differential given by
 $\delta+(-1)^iP_i(d_j)$ on $P_i(X_j)$.  There is a chain epimorphism
 $P(X)\to X$ given by $\epsilon\colon P_0(X_j)\to X_j$ on $P_0(X_j)$, and by
 zero on $P_i(X_j)$ for $i>0$.  We will just write $\epsilon$ for this.

 Now suppose we have a chain map $f\colon X\to Y$.  This gives maps
 $P_i(f_j)\colon P_i(X_j)\to P_i(Y_j)$ and by taking a direct sum we get a
 map $P(f)\colon P(X)\to P(Y)$ of graded objects.  We claim that this is
 actually a chain map, or in other words that the map
 $\phi=d^{PY}\circ P(f) - P(f)\circ d^{PX}$ is zero.  Indeed, on
 $P_i(X_j)$ the map $\phi$ consists of a component
 $\phi'\colon P_i(X_j)\to P_{i-1}(Y_j)$ and a component
 $\phi''\colon P_i(X_j)\to P_i(Y_{j-1})$, which are given by 
 \begin{align*}
  \phi' &= \delta^{Y_j}_i\circ P_i(f_j) - P_{i-1}(f_j)\circ\delta^{X_j}_i;\\
  \phi'' &= (-1)^iP_i(d^Y_j)\circ P_i(f_j) - 
            P_i(f_{j-1})\circ (-1)^iP_i(d^X_j).
 \end{align*}
 Naturality of $\delta$ implies that $\phi'=0$.  Functoriality of
 $P_i$, together with the fact that $f$ is a chain map, implies that
 $\phi''=0$, so $\phi=0$ as claimed.  Thus, $P$ gives a functor
 $\Ch(\A{\cat{U}})\to\Ch(\P{\cat{U}})$.
\end{Cons}

\begin{Rem}\label{rem-epsilon-qi}
 As $H_*(P_*(X_j))\cong H_0(P_*(X_j))\cong X_j$, the standard spectral sequence
 for a double complex collapses to give $H_*(P(X))\cong H_*(X)$.  More
 precisely, the map $\epsilon\colon P(X)\to X$ is a quasi-isomorphism.  
\end{Rem}

\begin{Rem}\label{rem-P-Gr}
 The graded object underlying $P(X)$ depends only on the graded object
 underlying $X$.  Thus, we have a commutative diagram 
 \begin{center}
  \begin{tikzcd}
   \Ch(\A{\cat{U}}) \arrow[r,"P"] \arrow[d,"\text{forget}"'] &
   \Ch(\P{\cat{U}}) \arrow[d,"\text{forget}"] \\
   \Gr(\A{\cat{U}}) \arrow[r,"P"'] & 
   \Gr(\P{\cat{U}}).
  \end{tikzcd}
 \end{center}
\end{Rem}

\begin{Prop}\label{prop-P-enriched}
 The functor $P\colon \Ch(\A{\cat{U}})\to\Ch(\P{\cat{U}})$ has a natural
 enrichment over $\Ch(\Ab)$.  In more detail, recalling \cref{rec-internal_homs}, for
 $X,Y\in\Ch(\A{\cat{U}})$ there is a morphism  
 \[ P \colon  \Hom(X,Y) \to \Hom(P(X),P(Y)) \]
 in $\Ch(\Ab)$ which is compatible with identities and composition, and
 whose effect on $\Hom_{\Ch(\A{\cat{U}})}(X,Y)=Z_0(\Hom(X,Y))$ is as
  given in \cref{cons-P}.
\end{Prop}
\begin{proof} 
 An element $f\in\Hom(X,Y)_n$ consists of a system of morphisms
 $f_j\colon X_j\to Y_{j+n}$ in $\A{\cat{U}}$ (see \cref{rec-internal_homs}).  The differential in $\Hom(X,Y)$ is given by 
 \[ d(f)_j = d^Y_{j+n}\circ f_j - (-1)^n f_{j-1}\circ d^X_j. \]
 Let $P(f)\in\Hom(P(X),P(Y))_n$ be the map given on $P_i(X_j)$ by
 \[ (-1)^{ni}P_i(f_j)\colon P_i(X_j)\to P_i(Y_{j+n}). \]  
 We will show that 
 $d(P(f))=P(d(f))$, by essentially the same argument that we used to
 prove that $P$ preserves chain maps.  On $P_i(X_j)$, the map
 $\phi=d(P(f))$ consists of a component 
 $\phi'\colon P_i(X_j)\to P_{i-1}(Y_{j+n})$ and a component 
 $\phi''\colon P_i(X_j)\to P_i(Y_{j+n-1})$, which are given by 
 \begin{align*}
  \phi' &= \delta^{Y_{j+n}}_i\circ (-1)^{ni}P_i(f_j) - 
           (-1)^n\cdot (-1)^{n(i-1)}P_{i-1}(f_j)\circ\delta^{X_j}_i \\
  \phi'' &= (-1)^iP_i(d^Y_{j+n})\circ(-1)^{ni}P_i(f_j) - 
           (-1)^n\cdot (-1)^{ni}P_i(f_{j-1})\circ(-1)^iP_i(d^X_j).
 \end{align*}
 Naturality of $\delta$ implies that $\phi'=0$.  Functoriality of $P_i$
 gives 
 \[ \phi'' = (-1)^{(n+1)i} 
     P_i\left(d^Y_{n+j}\circ f_j-(-1)^nf_{j-1}\circ d^X_j\right) =
     (-1)^{(n+1)i}P_i(d(f)_j).
 \]
 Because $d(f)\in\Hom(X,Y)_{n-1}$ and $(-1)^{(n+1)i}=(-1)^{(n-1)i}$ we
 see that $\phi''$ is the restriction of $P(d(f))$ to $P_i(X_j)$. This shows that $P$ is indeed a map of chain complexes and it is clear that it is compatible with compositions and identities. 
\end{proof}

Recall that the homotopy category of chain complexes has Hom-sets given by taking the zero homology of the internal hom objects, see \cref{def-K(B)}. Therefore as an application of the previous result we obtain the following.

\begin{Cor}\label{cor:pexists}
 If $f\colon X\to Y$ is nullhomotopic, then so is $P(f)\colon P(X)\to P(Y)$.
 Thus, $P$ induces a functor $\K{\A{\cat{U}}}\to\K{\P{\cat{U}}}$. \qed
\end{Cor}

\begin{Prop}\label{prop:pproperties}\leavevmode
Let $\cat U$ be as in \cref{hyp-basic}.
 \begin{enumerate}
  \item $P\colon\Ch(\A{\cat{U}})\to\Ch(\P{\cat{U}})$ 
   preserves quasi-isomorphisms.
  \item $P\colon\Ch(\A{\cat{U}})\to\Ch(\A{\cat{U}})$
   preserves finite limits and all colimits. 
  \item $P\colon\K{\A{\cat{U}}}\to\K{\A{\cat{U}}}$ 
   preserves exact triangles and all coproducts.
 \end{enumerate}
\end{Prop}
\begin{proof}\leavevmode
 \begin{itemize}
  \item[(a)] For any $f\colon X\to Y$ we have a commutative diagram 
   \begin{center}
    \begin{tikzcd}
     P(X) \arrow[d,"\epsilon"'] \arrow[r,"P(f)"] &
     P(Y) \arrow[d,"\epsilon"] \\
     X \arrow[r,"f"'] &
     Y
    \end{tikzcd}
   \end{center}
   in which the vertical maps are quasi-isomorphisms.  It follows
   immediately that $P(f)$ is a quasi-isomorphism if and only if $f$ is a
   quasi-isomorphism. 
  \item[(b)] This is immediate from \cref{lem-P-exact}.  
  \item[(c)] Let us say that a distinguished triangle
   $X\xrightarrow{f}Y\xrightarrow{g}Z\xrightarrow{h}\Sigma X$ is
   special if it has the form considered in
   \cref{rem-triangulation}, so $h=-udv$ for 
   some morphisms $X\xleftarrow{u}Y\xleftarrow{v}Z$ of graded objects
   with $uf=1_X$ and $fu+vg=1_Y$ and $gv=1_Z$.  We then have morphisms 
   $P(X)\xrightarrow{P(f)}P(Y)\xrightarrow{P(g)}P(Z)\xrightarrow{P(h)}\Sigma P(X)$ 
   of chain complexes and morphisms 
   $P(X)\xleftarrow{P(u)}P(Y)\xleftarrow{P(v)}P(Z)$ of graded objects
   with $P(u)P(f)=1_{P(X)}$ and $P(f)P(u)+P(v)P(g)=1_{P(Y)}$ and
   $P(g)P(v)=1_{P(Z)}$, so we see that $P$ preserves special 
   triangles.  Moreover, every distinguished triangle is linked by a system of
   quasi-isomorphisms to a special one, and $P$ preserves
   quasi-isomorphisms, so $P$ also preserves distinguished triangles. The claim about coproducts follows from (b) and the fact that coproducts in the homotopy category are constructed in chain complexes.\qedhere
 \end{itemize}
\end{proof}

\begin{Rem}\label{rem-sym-monoidal}
 We note that by
 \cref{prop-proj-closed-under-tensor} there is a non-unital symmetric monoidal
 structure on $\K{\P{\cat U}}$ making the inclusion
 $\K{\P{\cat U}} \to \K{\A{\cat U}}$ a non-unital symmetric
 monoidal functor. If $\cat U$ is such that $\unit \in \P{\cat{U}}$, then $\K{\P{\cat U}}$ is symmetric monoidal and the functor $\K{\P{\cat U}} \to \K{\A{\cat U}}$ refines to a symmetric monoidal functor. For
 example, this holds if $1 \in \cat{U}$, or see \cref{prop-minimal}
 for a characterisation of when $\unit$ is projective. 
\end{Rem}

We are finally ready to prove the main result of this section.

\begin{Thm}\label{thm-derived-category-equal-K-proj}
 Under \cref{hyp-basic}, the functor $P$ induces a non-unital symmetric monoidal equivalence of triangulated categories
 \[ \D{\cat{U}} \simeq \K{\P{\cat{U}}}. \]
 If furthermore $\unit \in \A{\cat U}$ is projective
 (\cref{prop-minimal}), then the above equivalence refines to a
 symmetric monoidal equivalence.  
\end{Thm}
\begin{proof}
 The functor $P \colon \K{\A{\cat U}} \to \K{\P{\cat U}}$ from \cref{cor:pexists} sends quasi-iso-morphisms to isomorphisms (see \cref{prop:pproperties}),
 and so induces a functor $\D{\cat{U}} \to \K{\P{\cat{U}}}$, which we
 still call $P$.  The inclusion $\K{\P{\cat{U}}}\to\K{\A{\cat{U}}}$
 induces a functor $J$ in the opposite direction and these are
 inverse. Note that the functor $J$
 is (possibly non-unital) symmetric monoidal by \cref{rem-sym-monoidal}, so the 
 inverse functor $P$ is also (possibly non-unital) symmetric monoidal.
\end{proof}
\begin{Rem}
 Some readers will prefer to work in the context of quasicategories  so we can make a quasicategorical formulation of essentially the same result as follows.
 As $\Ch(\A{\cat{U}})$ is enriched over $\Ch(\Ab)$, it can be regarded
 as a dg category.  The dg nerve construction of~\cite[Construction
 1.3.1.6]{HA} converts dg categories to quasicategories.  As $P$ is
 an enriched functor, it induces a functor 
 $N_{\dg}(\Ch(\A{\cat{U}}))\to N_{\dg}(\Ch(\P{\cat{U}}))$ of
 quasicategories, and similarly for $J$. The claim is that the functor $P\colon N_{\dg}(\Ch(\A{\cat{U}}))\to N_{\dg}(\Ch(\P{\cat{U}}))$ is a
 quasicategorical localization with respect to the quasi-isomorphisms.
 For a more detailed statement, consider another quasicategory
 $\cat{C}$, and let $\Fun_{qi}(N_{\dg}(\Ch(\A{\cat{U}})),\cat{C})$ be the full subcategory
 of $\Fun(N_{\dg}(\Ch(\A{\cat{U}})),\cat{C})$ consisting of the functors that send
 quasi-isomorphisms to isomorphisms.  The claim is then that 
 \[ P^*\colon \Fun(N_{\dg}(\Ch(\P{\cat{U}})),\cat{C})\to \Fun_{qi}(N_{\dg}(\Ch(\A{\cat{U}})),\cat{C}) \]
 is an equivalence of quasicategories.  Indeed, we have a functor 
 $J^*$ in the opposite direction, and a natural isomorphism
 $PJ\simeq 1$, and a natural quasi-isomorphism $JP\to 1$; the claim
 follows easily from these facts.
\end{Rem}

\section{Thin complexes}
\label{sec-thin}

In this section we show that any object in the derived category
$\D{\cat{U}}$ is equivalent to a complex of projective objects whose
differentials satisfy a special condition. The existence of such
replacements allows one to interpolate efficiently between the abelian
category $\A{\cat{U}}$ and its derived category $\D{\cat{U}}$.

We first introduce the special type of complexes we will consider in this section. To this end, recall the filtration from \cref{cons-filtration}.
 
\begin{Def}\label{defn-thin}
 We say that $X\in \Ch(\P{\cat U})$ is \emph{thin} if $d(L_{\leq s}X)\leq L_{<s}X$ for all $s$. Note that in this case the quotient complexes
 $L_sX=L_{\leq s}X/L_{<s}X$ have zero differential.  
\end{Def}

We now work towards proving that any complex of objects in $\A{\cat{U}}$ has a thin replacement which is unique up to isomorphism of complexes. 
\begin{Def}\label{defn-thin-replacement}
 Consider a complex $X\in\Ch(\A{\cat{U}})$.  A \emph{thin replacement} of
 $X$ consists of a thin complex $U\in\Ch(\P{\cat{U}})$ and a
 quasi-isomorphism $u\colon U\to X$. 
\end{Def}

\begin{Prop}\label{prop-thin}
 Let $X\in\Ch(\P{\cat U})$ be a complex of projective objects. Then 
 we can choose a decomposition of chain complexes $X\cong X'\oplus X''$ where $X'$ is thin and $X''$ is contractible. 
 In particular, $X$ is homotopy equivalent to the thin complex
 $X'$.  
\end{Prop}

 \begin{proof}
 For any $s\geq 1$ we have a quotient map $q_s\colon L_{\leq s}X\to L_sX$.
 By \Cref{prop-pure-split}, we can choose a splitting
 $L_sX=U(s)\oplus V(s)$ where $U(s)$ has zero differential and $V(s)$
 is contractible.  We know from \Cref{cor-acyc-proj} that
 $V(s)$ is projective in $\Ch(\A{\cat{U}})$, so we can choose a section of
 the projection $q_s^{-1}(V(s))\to V(s)$ and thus regard $V(s)$ as a
 subcomplex of $L_{\leq s}X$.  Similarly, $U(s)$ is projective in
 $\mathrm{Gr}(\A{\cat{U}})$ and so can be lifted to a graded subobject of
 $L_{\leq s}X$.  This need not be a subcomplex, but it will satisfy
 $d(U(s))\subseteq  L_{<s}X$.  It is easy to see that
 $L_{\leq s}X\cong U(s)\oplus V(s)\oplus L_{<s}X$, and we can extend this
 inductively to get
 $L_{\leq s}X\cong \bigoplus_{i\leq s}U(i)\oplus\bigoplus_{i\leq s}V(i)$.
 Thus, if we put $U=\bigoplus_{i>0}U(i)$ and $V=\bigoplus_{i>0}V(i)$ then
 $V$ is a contractible subcomplex of $X$ and $X\cong U\oplus V$.  It
 follows that the quotient $X/V$ is a thin complex that is isomorphic in $\mathrm{Gr}(\A{\cat{U}})$ to $U$. \Cref{cor-no-ext} tells us that
 $\Ext^1_{\Ch(\A{\cat{U}})}(X/V,V)=0$, so we must have
 $X\cong (X/V)\oplus V$ as chain complexes.
\end{proof}

\begin{Prop}\label{prop-thin-iso}
 Let $f\colon X\to Y$ be a quasi-isomorphism between thin complexes. Then
 $f$ is actually an isomorphism of chain complexes.
\end{Prop}
\begin{proof}
 The mapping cone $C(f)$ is an acyclic complex of projectives, so it is contractible, see \Cref{cor-acyc-proj}.  If we choose a contraction $\sigma$, we see that $L_s(\sigma)$ gives a contraction of the complex $L_s(C(f))\cong C(L_sf)$, so the map $L_sf\colon L_sX\to L_sY$ must again be a quasi-isomorphism.  As the
 differential on $L_sX$ and $L_sY$ is zero, this just means that
 $L_sf$ is an isomorphism.  It follows by an evident induction that
 $L_{\leq s}f$ is an isomorphism for all $s$, and we can pass to the
 colimit to see that $f$ is an isomorphism.
\end{proof}

\begin{Thm}\label{prop-thin-replacement}
 Any complex $X\in \Ch(\A{\cat{U}})$  has a thin replacement which is unique up to isomorphism of chain complexes.
\end{Thm}
\begin{proof}
 \cref{cons-P} gives a quasi-isomorphism $P(X)\to X$ with $P(X)$ a chain complex of projective objects. \cref{prop-thin} then gives a decomposition $P(X)\cong U\oplus T$ where $U$ is thin and $T$ is contractible.  The composite $u=(U\to P(X)\to X)$ is then a thin
 replacement.  Suppose that $(V\xrightarrow{v}X)$ is another thin
 replacement. As $U$ is a chain complex of projective objects and $v$ is a quasi-isomorphism,
 \cref{cor-cofibrant-alt} gives $f\colon U\to V$ with $vf$ chain
 homotopic to $u$.  As $u$ and $v$ are quasi-isomorphisms, we see that
 $f$ is also a quasi-isomorphism, and therefore an isomorphism by
 \cref{prop-thin-iso}.
\end{proof}

We end this section by discussing some other features of thin complexes.
\begin{Prop}\label{prop-thin-retract}
 Let $X$ and $Y$ be thin complexes, and suppose that $X$ is a
 retract of $Y$ in $\K{\P{\cat{U}}}$.  Then $X$ is already a retract of
 $Y$ in $\Ch(\P{\cat{U}})$.
\end{Prop}
\begin{proof}
 By hypothesis, we have maps $X\xrightarrow{f}Y\xrightarrow{g}X$ in $\Ch(\P{\cat{U}})$ such
 that the composite $u=gf$ is chain homotopic to the identity.  This
 means in particular that $u$ is a quasi-isomorphism and so is an
 isomorphism in $\Ch(\P{\cat{U}})$ by \cref{prop-thin-iso}.
 The map $h=u^{-1}g\colon Y\to X$ now
 satisfies $hf=1$ and so exhibits $X$ as a retract of $Y$ in
 $\Ch(\P{\cat{U}})$. 
\end{proof}

\begin{Rem}\label{rem-thin-tensor}
 The tensor product of two thin complexes need not be thin.  To see
 this, let $C=C_2$ be the cyclic group of order $2$, and let $\epsilon$ be the
 homomorphism $C\to 1$.  Let $X$ be the complex $e_C\xrightarrow{\epsilon}e_1$,
 which is thin.  The only wide subgroups in $C\times C$ are the full
 group and the diagonal subgroup $\Delta$, so
 $e_C\otimes e_C\cong e_{C^2}\oplus  e_\Delta$, with $e_\Delta\cong e_C$, see \cite[Remark 4.12]{PolStrickland2022}.  The chain
 complex $L_2(X\otimes X)$ is isomorphic to
 \[(e_C\xrightarrow{\bbm 1\\-1\ebm}e_C\oplus  e_C\xrightarrow{}0)\] so the differential
 is not trivial. 
\end{Rem}

\section{Compact objects}
\label{sec-compact}
In this section we study the compact objects in $\D{\cat{U}}$ and give various equivalent characterisations of them. Moreover, we show that the homology of a non-zero compact object must contain a torsion-free element whenever $\cat{U}$ is a multiplicative global family.  Recall that we are always working under \cref{hyp-basic}.

We first identify a set of compact generators.
\begin{Prop}\label{prop-comp-gen}
 The objects $\{e_G\mid G\in\cat{U}\}$ form an essentially small class of
 compact objects that generates $\D{\cat{U}}$ (equivalently $\K{\P{\cat{U}}}$).
\end{Prop}
\begin{proof}
 We have a natural isomorphism $\Hom_{\D{\cat{U}}}(\Sigma^i e_G,X)\cong H_i(X(G))$ for each $i \in \Z$. This implies that $e_G$ is compact for each $G \in \cat{U}$. Moreover, if 
 \[
 \Hom_{\D{\cat{U}}}(\Sigma^i e_G, X) \cong H_i(X(G))= 0
 \]
 for each $i \in \Z$, and $G \in \cat{U}$, then $X \simeq 0$ so the set generates $\D{\cat{U}}$. 
\end{proof}

\begin{Def}\label{defn-perfect}
 A chain complex $X\in\Ch(\A{\cat{U}})$ is \emph{perfect} if it is a bounded complex of finitely generated projectives. This means
 that $X$ is isomorphic in $\Gr(\A{\cat{U}})$ to
 $\bigoplus_{i=1}^r\Sigma^{n_i}e_{G_i,V_i}$ for some $n_i \in \bbZ$,
 groups $G_i\in\cat{U}$, and non-trivial finite-dimensional
 $k[\Out(G_i)]$-representations $V_i$ by \cref{rem-projective-objects}.
\end{Def}

We now give a characterisation of the compact objects in $\D{\cat{U}}$. Some of the following equivalent conditions could be obtained from general results such as~\cite[tag 094B]{stacks}, but the particular features of $\A{\cat{U}}$ allow for a simpler and more direct proof.

\begin{Thm}\label{prop-perfect}
 Let $\cat U$ be a family satisfying \cref{hyp-basic}. For a chain complex $X\in\Ch(\A{\cat{U}})$, the following are equivalent:
 \begin{itemize}
  \item[(a)] $X$ is isomorphic in $\Ch(\A{\cat{U}})$ to the direct sum of a
   thin perfect complex and a contractible complex;
  \item[(b)] $X$ is isomorphic in $\D{\cat{U}}$ to a thin perfect complex;
  \item[(c)] $X$ is isomorphic in $\D{\cat{U}}$ to a perfect complex;
  \item[(d)] $X$ is a retract in $\D{\cat{U}}$ of a perfect complex;
  \item[(e)] $X$ lies in the thick subcategory of $\D{\cat{U}}$ 
   generated by $\{e_G\mid G\in\cat{U}\}$;
  \item[(f)] $X$ is a compact object of $\D{\cat{U}}$.
 \end{itemize}
\end{Thm}
\begin{proof}
 It is immediate that (a)$\Rightarrow$(b)$\Rightarrow$(c)$\Rightarrow$(d).  We next
 prove that (d)$\Rightarrow$(a). By \cref{thm-derived-category-equal-K-proj} we can work in $\K{\P{\cat{U}}}$ instead of $\D{\cat{U}}$, so we suppose that $X\in\K{\P{\cat{U}}}$ and that
 $X$ is a retract in $\K{\P{\cat{U}}}$ of some perfect complex $Y$.  By
 \cref{prop-thin}, we can choose isomorphisms
 $X \cong X'\oplus  X''$ and $Y \cong Y'\oplus  Y''$ in $\Ch(\P{\cat{U}})$ such that $X'$
 and $Y'$ are thin and $X''$ and $Y''$ are contractible.  As $Y'$ is a
 retract of $Y$ in $\Ch(\P{\cat{U}})$, we see that $Y'$ is also perfect.
 In $\K{\P{\cat{U}}}$ we have $X\simeq X'$ and $Y\simeq Y'$ so $X'$ is a
 retract of $Y'$.  \cref{prop-thin-retract} now tells us
 that $X'$ is already a retract of $Y'$ in $\Ch(\P{\cat{U}})$, so $X'$ is a
 thin perfect complex.  This completes the proof that
 (d)$\Rightarrow$(a), so (a) to (d) are all equivalent.  

 If $X$ is a thin perfect complex, then it is equal to $L_{\leq s}X$
 for sufficiently large $s$, and each filtration quotient $L_nX$ is isomorphic as a chain
 complex to a retract of a finite direct sum of terms of the form $\Sigma^me_{G,V}$ where $V$ is finite dimensional. Write $V \cong \oplus_{i=1}^l S_i$ as a sum of indecomposables so that $e_{G,V} \cong \oplus_{i=1}^l e_{G,S_i}$. Since each $e_{G,S_i}$ is a retract of $e_G$, we see that $e_{G,V} \in \thick{e_G}$. Then by induction using the distinguished triangle $L_{<n}X \to L_{\leq n}X \to L_nX$, we conclude that (b)$\Rightarrow$(e).  
 
 Conversely, let $\cat{S}$ be the
 full subcategory where conditions (a) to (d) are satisfied.  As the
 mapping cone of any morphism of perfect complexes is again a perfect
 complex, it is not hard to check that $\cat{S}$ is thick, and it clearly
 contains $e_G$ for all $G \in \cat{U}$.  This proves that (e) implies (a) to (d),
 and thus that (a) to (e) are all equivalent.

 All that is left to prove is that (e) and (f) are equivalent, which is immediate from \cite[Lemma 2.2]{Neemanlocalization} together with \cref{prop-comp-gen}.
\end{proof}

We next give a concrete description of the compact objects in $\D{\cat{U}}$ in the case when $\cat{U}$ contains only finitely many isomorphism classes of groups.
\begin{Prop}\label{prop-finite-U}
 Suppose that $\cat{U}$ is essentially finite (i.e., it has only finitely
 many isomorphism classes of objects).  Then an object $X\in\D{\cat{U}}$ is
 compact if and only if for every $G\in\cat{U}$, the sum
 $\bigoplus_{n\in\Z}H_n(X(G))$ has finite dimension.
\end{Prop}
\begin{proof}
 First suppose that $X$ is compact.  By
 \cref{prop-perfect}, we can assume that $X$ is a thin
 perfect complex, so $X_n$ is finitely generated projective for all $n$, and
 is zero for all but finitely many $n$.  For any $G$ it then follows
 that $X_n(G)$ is finite dimensional  for all $n$, and is zero for all
 but finitely many $n$.  It follows in turn that $H_*(X)$ has the same
 finiteness properties. 

 For the converse suppose that $\bigoplus_{n\in\Z}H_n(X(G))$ has 
 finite dimension. Without loss of generality, by \cref{prop-thin-replacement} we can assume that 
 $X$ is thin. We will prove by induction on 
 $s\geq 0$ that $L_{\leq s}X$ is compact in $\D{\cat{U}}$. The backwards 
 implication will then follow by noting that $X=L_{\leq s}X$ for some 
 $s\gg 0$ as $\cat{U}$ is essentially finite. The base of the induction is clear since 
 $L_{\leq 0}X=0$. Suppose that $s>0$ and consider the short exact sequence 
 $L_{<s}X\to L_{\leq s}X\to L_sX$. Since $L_{<s}X$ is compact by induction, 
 it suffices to show that $L_sX$ is compact. 
 
 Since $X$ was assumed to be 
 thin, we have $L_s X\cong \bigoplus_{i=1}^r \Sigma^{n_i} e_{G_i,V_i}$ for 
 some family of representations $V_i$ and groups $G_i\in\cat{U}_{=s}$. 
 Note that since $\cat{U}$ is essentially finite, we can arrange the direct sum to be finite 
 but a priori the representations $V_i$ might be infinite dimensional.  
 It follows that $L_sX$ is compact if and only if 
 $L_s X(G)$ is finite dimensional for all $G\in \cat{U}_{=s}$. Since $L_s X$ has no 
 non-zero differentials, this is also equivalent to 
 $H_n(L_s X(G))$ being finite dimensional for all $G\in \cat{U}_{=s}$ and $n\in\Z$.
 By our inductive hypothesis and the first paragraph we know that 
 $\bigoplus_{n\in\Z}H_n(L_{<s}X(G))$ is finite dimensional for all $G\in\cat{U}_{=s}$. 
 Moreover, $L_{\leq s}X(G)=X(G)$ for all $G\in\cat{U}_{=s}$ by definition of 
 $L_{\leq s}$. This together with our initial assumption implies that 
 $\bigoplus_n H_n(L_{\leq s}X(G))$ is finite dimensional 
 for all $G\in\cat{U}_{=s}$. The long exact sequence in homology then shows that 
 $H_n(L_sX(G))$ is finite dimensional for all $G\in\cat{U}_{=s}$ and $n\in\Z$, 
 concluding the proof. 
\end{proof}

We now give an application of thin complexes to the structure of the homology of compact objects for multiplicative global families. This result is a key piece in understanding the tt-geometry of $\D{\cat{U}}$ as we will show in the sequel~\cite{BBPSWttgeometry}. 
First we recall the following definition from \cite[Definition 12.1]{PolStrickland2022}.
\begin{Def}
    Consider $X \in \A{\cat U}$ and $x \in X(G)$ for some $G \in \cat U$. 
    \begin{enumerate}
        \item We say that $x$ is \emph{torsion} if there exists $\alpha\colon H \to G$ in $\cat U$ such that $\alpha^*(x)=0$.
        \item We say that $x$ is \emph{torsion-free} if it is not torsion. 
        \item We say that $X$ is \emph{torsion} if it consists of torsion elements. 
        \item We say that $X$ is \emph{torsion-free} if it does not contain any non-zero torsion element. Equivalently, $X$ is torsion-free if and only if all the maps $\alpha^*\colon X(G) \to X(H)$ are injective.
    \end{enumerate}
\end{Def}
\begin{Exa}
    The objects $e_{G,V}\in \A{\cat U}$ are torsion-free by~\cite[Lemma 12.10]{PolStrickland2022}.
\end{Exa}
\begin{Thm}\label{thm-compact-torsion-free-element}
 Let $\cat{U}$ be a multiplicative global family. If $X \in \D{\cat{U}}$ is non-zero and compact, then there exists a torsion-free element in $H_n(X)$ for some $n \in \Z$. 
\end{Thm}
\begin{proof}
 By \cref{prop-perfect}, we can assume that $X$ is a non-zero thin perfect complex. Choose $m$ minimal so that $X=L_{\leq m}X$ and $G$ with $|G|=m$ such that $X$ contains a summand of the form $\Sigma^de_{G,S}$. 
 Note that $L_{<m}X$ contains only summands $e_{H,T}$
 with $|H|<m$ so $e_G(H)=0$ so $\Hom_{\A{\cat{U}}}(e_{H,T},e_{G,S})=0$.  This
 shows that the complex $\Hom_{\A{\cat{U}}}(X,e_{G,S})$ can be identified with the
 complex $\Hom_{\A{\cat{U}}}(L_mX,e_{G,S})$, in which the differential is zero.
 It follows that the projection $(L_mX)_d \cong X_d\to e_{G,S}$ gives a non-trivial
 class in $H_d(\Hom_{\A{\cat{U}}}(X,e_{G,S}))$.  As $e_{G,S}$ is injective for multiplicative global families $\cat{U}$ by \cite[Proposition 15.1]{PolStrickland2022}, we can
 also identify $H_d(\Hom_{\A{\cat{U}}}(X,e_{G,S}))$ with
 $\Hom_{\A{\cat{U}}}(H_{-d}(X),e_{G,S})$.  We thus have a non-trivial map
 $f\colon H_{-d}(X)\to e_{G,S}$.  This means that there exists $G'\in\cat{U}$ and
 $x\in H_{-d}(X)(G')$ such that $f(x)$ is non-zero in $e_{G,S}(G')$.  The
 object $e_{G,S}$ is torsion-free so
 $x$ is not a torsion element.
\end{proof}
We note that the theorem does not hold if we relax the assumption on the family.
\begin{Exa}\label{ex-cyclic-groups-no-growth}
    Let $\cat U$ be the family of cyclic $2$-groups. Let $X=\mathrm{cof}(e_{C_2}\to \unit )\in\D{\cat U}$ which is non-zero and compact. We can easily calculate that 
    \[
    H_*(X(C_{2^n}))\cong H_0(X(C_{2^n}))\cong\begin{cases}
        0 & \mathrm{if}\; n>0 \\
        k & \mathrm{if}\; n=0.
    \end{cases}
    \]
    Thus, in this case $X$ does not contain any torsion-free element.
\end{Exa}

\begin{Rem}\label{rem-growth}
 The existence of torsion-free classes in the homology of any non-zero compact object can interpreted as encoding certain growth properties in $\cat U$ as we now explain. Consider a morphism $\alpha\colon G \to H$ in $\cat U$ which is not an isomorphism, and consider 
 \[
 X=\mathrm{cof}(e_G \xrightarrow{e_{\alpha}}e_H)\in \D{\cat U}^c.
 \]
 The long exact sequence in homology takes the form 
 \[
 0 \to H_1(X) \to e_G \xrightarrow{e_{\alpha}} e_H \to H_0(X) \to 0
 \]
 from which we deduce that $H_1(X)$ is torsion-free (as it is a subobject of a torsion-free object) provided that $H_1(X)\not =0$. This latter condition is ensured by requiring that the map $e_\alpha$ is not injective, or equivalently that $|\Hom_{\cat U}(T,G)|> |\Hom_{\cat U}(T,H)|$ for some $T \in \cat U$. For more general $X\in \thick{e_G,e_H}$, the existence of a torsion-free element in $H_*(X)$ is ensured by requiring that the growth rate of $|\Hom_{\cat U}(T,G)|$ is greater than that of $|\Hom_{\cat U}(T,H)|$ as $T$ increases. Subcategories 
$\cat U$ exhibiting this type of well-behaved growth rate for their groups were first studied in \cite[Section 10]{PolStrickland2022}, where they were referred to as complete subcategories. Interestingly, any multiplicative global family is complete \cite[Proposition 10.7]{PolStrickland2022}, while the family of cyclic $p$-groups is not~\cite[Example 10.6]{PolStrickland2022}. This is inline with \cref{ex-cyclic-groups-no-growth}.
\end{Rem}

The existence of a torsion-free element in the homology of a non-zero compact object $X$ will let us prove that the thick ideal generated by $X$ contains a generator $e_G$ for some $G \in \cat U$ when $\cat{U}$ is a multiplicative global family. This will be an important ingredient in our companion paper \cite{BBPSWttgeometry}. Before proving this result we need the following lemma.
\begin{Lem}\label{lem-eG-eGX}
 Let $\cat U$ be a multiplicative global family and consider $X\in\A{\cat U}$ and $x\in X(G)$.  Suppose that $x$ is a
 torsion-free element.  Then $x$ gives rise to a split monomorphism
 $i\colon e_G\to e_G\otimes X$. 
\end{Lem}
\begin{proof}
 For any $G'\in\CU$ we define 
 \[ 
 i_{G'}\colon e_G(G')\cong k[\Hom_{\cat{U}}(G',G)]\to k[\Hom_{\cat{U}}(G',G)]\otimes X(G') 
 \]
 by $i_{G'}([\alpha])=[\alpha]\otimes\alpha^*(x)$.  This gives a morphism
 $i\colon e_G\to e_G\otimes X$ in $\A{\cat U}$.  Next, for each morphism
 $\alpha \colon G'\to G$ in $\cat U$, it is given that the element
 $\alpha^*(x)\in X(G')$ is non-zero, so we can choose a $k$-linear map
 $p_\alpha\colon X(G')\to k$ with $p_\alpha(\alpha^*(x))=1$.  Using these maps, we can define 
 \[ 
 p_{G'}\colon k[\Hom_{\cat{U}}(G',G)]\otimes X(G')\to k[\Hom_{\cat{U}}(G',G)] 
 \]
 by $p_{G'}([\alpha]\otimes v)=p_\alpha(v)[\alpha]$.  This satisfies
 $p_{G'}\circ i_{G'}=1$, so $i$ is a monomorphism.  The maps $p_{G'}$
 are not generally natural in $G'$, so we cannot immediately conclude
 that $i$ is a split monomorphism.  However, we know that $e_G$
 is an injective object in $\A{\cat U}$ by \cite[Proposition 15.1]{PolStrickland2022}, and this gives a splitting
 automatically.  
\end{proof}

\begin{Cor}\label{cor-eG-retract}
 Let $\cat U$ be a multiplicative global family and let $X\in\D{\cat U}^c$ be non-zero. Then there exists $G\in \cat U$ such that 
 $e_G\in\thickt{X}$, i.e., $e_G$ is contained in the thick ideal generated by $X$.
\end{Cor}

\begin{proof}
    By \cref{thm-compact-torsion-free-element} there exists a torsion-free element $x \in H_n(X)(G)$ for some $G \in \cat U$ and some $n \in \bbZ$. As $e_G$ is both injective and projective (see \cite[Proposition 15.1]{PolStrickland2022}), we have 
 \[
 \Hom_{\D{\cat U}}(\Sigma^n e_G,e_G\otimes X)\cong\Hom_{\A{\cat U}}(e_G,H_n(e_G\otimes X))
 \]
 and 
 \[
 \Hom_{\D{\cat U}}(e_G\otimes X,\Sigma^ne_G)\cong\Hom_{\A{\cat U}}(H_n(e_G\otimes X),e_G).
 \]
 It will
 therefore suffice to show that $e_G$ is a retract of the object
 $H_n(e_G\otimes X)\cong e_G\otimes H_n(X)$, and this follows from
 \cref{lem-eG-eGX}.  
\end{proof}

\section{Dualizable objects}\label{sec:dualisable}
In this section, we study the dualizable objects in $\D{\cat{U}}$, and show that $\D{\cat{U}}$ is rigidly-compactly generated if and only if $\cat{U}$ is a finite groupoid, see \cref{groupoidrigid}. We remind the reader that we are always working under \cref{hyp-basic}.

\begin{Def}
An object $X \in \D{\cat{U}}$ is \emph{dualizable} if the natural map 
\[
\iHom(X,\unit) \otimes Y \to \iHom(X,Y)
\]
is an isomorphism for all $Y \in \D{\cat{U}}$. We write $\D{\cat{U}}^{\dual}\subseteq \D{\cat{U}}$ for the full subcategory on the dualizable objects.
\end{Def}

\begin{Rem}
Let us write $\mathbb{D} = \iHom(-,\unit)$ for the functional dual functor. If $X$ is dualizable, then the right adjoint of $X \otimes -$ has the form $\mathbb{D}X \otimes -$.
\end{Rem}
\begin{Lem}\label{compactsaredualizable}
    Suppose that $\unit \in \A{\cat U}$ is finitely generated projective (e.g. if $1\in \cat U$ or $\cat U$ is a finite groupoid). If $X \in \D{\cat{U}}$ is dualizable, then it is compact.
\end{Lem}
\begin{proof}
    Note that the unit $\unit$ is perfect and hence compact by \cref{prop-perfect}. Therefore if $X$ is dualizable, then \[\Hom_{\D{\cat{U}}}(X, \bigoplus Y_i) \cong \Hom_{\D{\cat{U}}}(\unit, \bigoplus \mathbb{D}X \otimes Y_i) \cong \bigoplus\Hom_{\D{\cat{U}}}(X, Y_i)\] showing that $X$ is compact.
\end{proof}

We now identify the subcategory of dualizable objects when $\cat{U}$ contains the trivial group.
\begin{Lem}\label{nonzeroat1}
    If $1\in \cat{U}$, and $X \in \D{\cat{U}}$ is non-zero and dualizable, then $H_*(X)(1)$ is non-zero.
\end{Lem}
\begin{proof}
    As $X$ is dualizable and the evaluation functor preserves quasi-isomorphisms, we have 
    \[
    X(1) \otimes \mathbb{D}(X)(1)\cong \iHom(X,X)(1) \quad \mathrm{in}\; \D{\Mod{\bbQ}}.
    \]
    It is then enough to show that the right hand side is not acyclic. However  
    \[
    \iHom(X,X)_0(1)\cong \Hom_{\D{\cat U}}(X,X)\not =0
    \]
    since $X$ is non-zero.
\end{proof}

\begin{Prop}\label{dualisthick1}
 Suppose that $1\in\cat{U}$, and let $\Gr^f(\Mod{k})$ denote the
 category of graded vector spaces of finite total dimension over $k$.
 Then the functor $X\mapsto H_*(X)(1)$ gives an equivalence of categories from
 $\D{\cat{U}}^{\dual}$ to $\Gr^f(\Mod{k})$, and
 $\D{\cat{U}}^{\dual}=\thick{\unit}$.
\end{Prop}
\begin{proof}
 We have a functor $F(X)=H_*(X)(1)$ from $\D{\cat{U}}$ to
 $\Gr(\Mod{k})$ and a functor $C$ in the opposite direction sending
 $V$ to the constant functor with value $V$.  Both of these are
 strongly symmetric monoidal, so they preserve dualizability and
 restrict to give functors 
 $F\colon\D{\cat{U}}^{\dual}\to\Gr^f(\Mod{k})$ and 
 $C\colon\Gr^f(\Mod{k})\to\D{\cat{U}}^{\dual}$.  It is clear that
 $FC\simeq 1$. \cref{nonzeroat1} says that if $X$ is dualizable and $F(X)=0$ then $X\simeq 0$.  By applying this to cofibres, we see that $F\colon\D{\cat{U}}^{\dual}\to\Gr^f(\Mod{k})$ reflects isomorphisms. 
 
 Now suppose we have $X\in\D{\cat{U}}^{\dual}$.  We can choose bases for the vector spaces $H_i(X)(1)$ and then choose representing cycles in $Z_i(X)(1)$; these combine to give a morphism $f\colon CF(X)\to X$ such that $F(f)$ is an isomorphism.  By the previous paragraph, $f$ itself is an isomorphism. It follows that $F$ and $C$ are mutually inverse equivalences.  Finally, since the image of $C$ lies in $\thick{\unit}$, we see that $\thick{\unit}$ is contained in $\D{\cat{U}}^{\dual}$.  From this it follows that $\D{\cat{U}}^{\dual}=\thick{\unit}$.
\end{proof}

\begin{Rem}
    In particular, we see that $\D{\cat{U}}$ is seldom rigidly-compactly generated. In fact, we can make this more precise via the following theorem.
\end{Rem}

\begin{Thm}\label{groupoidrigid}
    Under \cref{hyp-basic}, the category $\D{\cat{U}}$ is rigidly-compactly generated if and only if $\cat{U}$ is a finite groupoid.
\end{Thm}
\begin{proof}
    Suppose that $\cat{U}$ is a finite groupoid. By \cref{compactsaredualizable}, we know that any dualizable object is compact, so we just need to verify that any compact object is dualizable. By \cref{prop-perfect}, any compact object is in the thick subcategory generated by $\{e_G \mid G \in \cat{U}\}$, so by a thick subcategory argument we can reduce to checking that each $e_G$ is dualizable. There is a symmetric monoidal equivalence of abelian categories 
    \begin{equation}\label{product}
    \A{\cat{U}} \simeq \prod_{G \in \pi_0\cat{U}} \Mod{k[\Out(G)]}.
    \end{equation}
    Accordingly, we may reduce to a single $G \in \pi_0\cat{U}$ and check that the canonical map
\[\Hom_k(k[\Out(G)], k) \otimes_{k} k[\Out(G)] \to \Hom_k(k[\Out(G)], k[\Out(G)])\]
is an isomorphism. This is clear since $k[\Out(G)]$ is finite dimensional over $k$.

Conversely, suppose that $\cat{U}$ is not a finite groupoid. There are two possibilities: either $\cat U$ is an infinite groupoid, or $\cat U$ is not a groupoid. In the first case, we still have a decomposition as in \eqref{product} from which we see that $\unit \in \A{\cat U}$ is projective but not finitely generated. By \cref{rem-projective-objects} this means that we can write $\unit$ as an infinite sum of the form  $\bigoplus_{i}e_{G_i, V_i}$. We now claim that $\unit \in \D{\cat U}$ is not compact, showing that $\D{\cat U}$ is not rigidly-compactly generated. Indeed if it were, then the identity map $\unit \to \unit \cong \bigoplus_{i}e_{G_i, V_i}$ would factor through a finite sum, say $\unit \hookrightarrow \bigoplus_{j=1}^n e_{G_{i_j}, V_{i_j}} \twoheadrightarrow \unit$, contradicting that $\unit$ is not finitely generated. 
For the second case, suppose that $\cat U$ is not a groupoid so that there is a morphism $\alpha\colon G \to H$ in $\cat{U}$ which is not an isomorphism. We will show that
\[\mathbb{D}e_G \otimes e_G \to \iHom(e_G, e_G)\] is not an isomorphism in $\A{\cat U}$, and this will imply that $e_G$ is not dualizable in $\D{\cat U}$. Evaluating at $H$, we see that
$(\mathbb{D}e_G \otimes e_G)(H) = 0$. On the other hand, we claim that
$\iHom(e_G,e_G)(H) \cong \Hom_{\A{\cat{U}}}(e_G \otimes e_H,e_G) \neq
0$. Let $\Gamma$ be the graph of $\alpha$ and note that
$\Gamma \cong G \in \cat{U}$. Write $W$ for the Weyl group of
$\Gamma$ in $G \times H$. By \cite[Proposition
4.11]{PolStrickland2022}, $e_\Gamma^W$ is a retract of
$e_G \otimes e_H$. Therefore the object
$\Hom_{\A{\cat{U}}}(e_\Gamma^W, e_G) \cong e_G(\Gamma)/W \neq 0$ is a
retract of $\Hom_{\A{\cat{U}}}(e_G \otimes e_H, e_G)$ and hence the
latter is also non-zero. Therefore $\D{\cat{U}}$ is not
rigidly-compactly generated.
\end{proof}

\begin{Rem}
    The characterisation of dualizable objects proved in \cref{dualisthick1} need not hold if $\cat{U}$ does not contain the trivial group. For example, suppose that $\cat{U}$ is a groupoid which does not contain the trivial group. Then the same argument as at the beginning of the proof of \cref{groupoidrigid} shows that each standard projective generator $e_G$ is dualizable, but none of them lie in $\thick{\unit}.$ Indeed, since all objects in $\A{\cat{U}}$ are projective, every complex is homotopy equivalent to its homology by \cref{lem-semisimple-complex}, so there is a symmetric monoidal equivalence of categories
    \[
    \D{\cat{U}} \xrightarrow{\sim} \prod_{G \in \pi_0\cat U}\Gr(\Mod{k[\Out(G)]}), \quad X \mapsto (H_*(X)(G))_G,
    \]
    and the claim then follows easily from this (as the target contains representations which are not necessarily trivial).
\end{Rem}

\section{The projective model structure}
\label{sec-proj-model}
In this section we prove the existence of the projective model structure on the category of chain complexes. 
One could also obtain this model structure by applying \cite[Theorem
11.6.1]{Hirschhorn} or the results of \cite{projclass}, but the
particular features of $\A{\cat{U}}$ allow us to give a self-contained
proof with explicit factorisations; it seems worthwhile to have an
example of this type in the literature. Recall that \cref{hyp-basic} is in place throughout the whole paper.

\begin{Def}\label{defn-model-structure}
 Consider a morphism $f\colon X\to Y$ in $\Ch(\A{\cat{U}})$.
 \begin{itemize}
  \item[(a)] We say that $f$ is a \emph{weak equivalence} if it is a
   quasi-isomorphism.
  \item[(b)] We say that $f$ is a \emph{cofibration} if it is a
   monomorphism, and $\cok(f)_n$ is projective in $\A{\cat{U}}$ for all
   $n$.
  \item[(c)] We say that $f$ is a \emph{fibration} if it is an
   epimorphism.
  \item[(d)] In view of \cref{cor-acyc-proj}, $f$ is an acyclic cofibration if and only if it is a monomorphism and $\cok(f)$ is a
   contractible complex of projective objects.
  \item[(e)] Similarly, $f$ is an acyclic fibration if it is an
   epimorphism with acyclic kernel.
 \end{itemize}
 We write $\we$, $\cof$, $\fib$, $\acf$ and $\afb$ for the five
 classes of maps defined above. We also use the symbols $\twoheadrightarrow$ and $\rightarrowtail$ to denote surjections and injections, respectively.
\end{Def}

We will show that the above definition gives a Quillen
model structure  that is monoidal, cofibrantly
generated and proper. We will
follow~\cite{Hovey} for the general theory of model categories; in
particular, our treatment of the axioms follows \cite[Definition 1.1.3]{Hovey}, using functorial replacements.
For the next definitions recall the functor $P$ from \cref{cons-P}, together with the chain map $\epsilon\colon PX \to X$.
\begin{Def}\label{defn-factor-M}
 Given a morphism $f\colon X\to Y$ in $\Ch(\A{\cat{U}})$, we define a complex $Mf$ by
 \begin{align*}
  (Mf)_n &= X_n\oplus  (PX)_{n-1}\oplus  (PY)_n & 
    d &= \bbm d & \epsilon & 0 \\ 0 & -d & 0 \\ 0 & -Pf & d \ebm .
\end{align*}
Since $\epsilon$ and $Pf$ are chain maps, one easily verifies that
$d^2 = 0$ so that $Mf$ is indeed a chain complex.  We also define
 \begin{align*}
  i_n &= \bbm 1 \\ 0 \\ 0 \ebm \colon  
   X_n \to X_n\oplus  (PX)_{n-1}\oplus  (PY)_n = (Mf)_n \\
  p_n &= \bbm f & 0 & \epsilon \ebm \colon  
   (Mf)_n = X_n\oplus  (PX)_{n-1}\oplus  (PY)_n \to Y_n.
 \end{align*}
\end{Def}

\begin{Prop}\label{prop-factor-M}
 Let $f\colon X \to Y$ in $\Ch(\A{\cat{U}})$. The maps
 $X\xrightarrow{i}Mf\xrightarrow{p}Y$ are chain maps, with $pi=f$.
 Moreover, $i$ is a cofibration and $p$ is an acyclic fibration.
\end{Prop}
\begin{proof}
Using the naturality of $\epsilon$ as well as
 the fact that $f$ and $\epsilon$ are chain maps, we have 
 \[ dp-pd = 
     \bbm df-fd & \epsilon(Pf)-f\epsilon & d\epsilon - \epsilon d \ebm = 0,
 \]
 so $p$ is a chain map.  It is clear that $i$ is a chain map and that
 $pi=f$.  It is also clear that $i$ is a monomorphism with
 $\cok(i)_n=(PX)_{n-1}\oplus (PY)_n$, so $i$ is a cofibration.  As
 $\epsilon\colon PY\to Y$ is an epimorphism, the same is true of $p\colon Mf\to Y$.
 All that is left is to prove that $\ker(p)$ is acyclic, or
 equivalently, that $p$ is a quasi-isomorphism.  For this, we introduce
 intermediate complexes $A$ and $B$ as follows:
 \begin{align*}
  A_n &= X_n\oplus  (PX)_{n-1}\oplus  Y_n & 
    d &= \bbm d & \epsilon & 0 \\ 0 & -d & 0 \\ 0 & -f\epsilon & d \ebm \\
  B_n &= X_n\oplus  X_{n-1}\oplus  Y_n & 
    d &= \bbm d & 1 & 0 \\ 0 & -d & 0 \\ 0 & -f & d \ebm.
 \end{align*}
 We find that $p$ can be factored as 
 \[ Mf \xrightarrow{p_1} A \xrightarrow{p_2} B \xrightarrow{p_3} Y, \]
 where $p_1=1\oplus 1\oplus \epsilon$ and $p_2=1\oplus \epsilon\oplus  1$ and 
 $p_3=\bbm f&0&1\ebm$.  All three of these maps are epimorphisms.
 For the first two, we have $\ker(p_1)\cong\ker(\epsilon_Y)$ and
 $\ker(p_2)\cong\Sigma\ker(\epsilon_X)$.  As $\epsilon_X$ and $\epsilon_Y$ are
 surjective quasi-isomorphisms, these kernels are acyclic, so $p_1$ and
 $p_2$ are quasi-isomorphisms.  For the third, we note that the
 inclusion $j\colon Y\to B$ is a chain map, with $p_3j=1$.  We can define 
 \[ s_n = \bbm 0&0&0 \\ 1&0&0 \\ 0&0&0 \ebm \colon  
     B_n = X_n\oplus  X_{n-1}\oplus  Y_n \to 
           X_{n+1}\oplus  X_n\oplus  Y_{n+1} = B_{n+1},
 \]
 and we find that 
 \[ ds+sd = \bbm 1&0&0 \\ 0&1&0 \\ -f&0&0 \ebm = 1-jp_3, \]
 so $p_3$ is a homotopy equivalence and therefore a
 quasi-isomorphism. 
\end{proof}

\begin{Def}\label{defn-factor-N}
 Given a morphism $f\colon X\to Y$ in $\Ch(\A{\cat{U}})$, we define a
 complex $Nf$ by
 \begin{align*}
  (Nf)_n &= X_n\oplus  (PY)_{n+1}\oplus  (PY)_n & 
    d &= \bbm d & 0 & 0 \\ 0 & -d & 1 \\ 0 & 0 & d \ebm .
    \end{align*}
    One easily checks that $d^2 = 0$ so that the above does define a chain complex. We also define
 \begin{align*}
  j_n &= \bbm 1 \\ 0 \\ 0 \ebm \colon  
   X_n \to X_n\oplus  (PY)_{n+1}\oplus  (PY)_n = (Nf)_n \\
  q_n &= \bbm f & 0 & \epsilon \ebm \colon  
   (Nf)_n = X_n\oplus  (PY)_{n+1}\oplus  (PY)_n \to Y_n.
 \end{align*}
\end{Def}
\begin{Prop}\label{prop-factor-N}
 Let $f\colon X \to Y$ in $\Ch(\A{\cat{U}})$.  The maps
 $X\xrightarrow{j}Nf\xrightarrow{q}Y$ are chain maps, with $qj=f$.
 Moreover, $j$ is an acyclic cofibration and $q$ is a fibration.
\end{Prop}
\begin{proof}
 Using the fact that $f$ and $\epsilon$ are chain maps, it is
 straightforward to check that $dj=jd$, $dq=qd$ and
 $qj=f$. As $\epsilon$ is an epimorphism, we see that the same is true of
 $q$, so $q$ is a fibration.  It is clear that $j$ is a monomorphism
 with $\cok(j)_n=(PY)_{n+1}\oplus (PY)_n$, which is projective.  The
 differential on $\cok(j)$ is $\bbm -d&1\\0&d\ebm$, so the 
 matrix $s_n=\bbm 0&0\\1&0\ebm$ gives a map
 $\cok(j)_n\to\cok(j)_{n+1}$ with $ds+sd=1$, proving that $\cok(j)$ is
 contractible. 
\end{proof}

We will also need some facts about lifting properties that are most
naturally formulated in a general abelian category.

\begin{Prop}\label{prop-lift-abelian}
 Let $\cat{A}$ be an abelian category, and let $A\xrightarrow{i}B\xrightarrow{p}C$ and
 $K\xrightarrow{j}L\xrightarrow{q}M$ be short exact sequences.  For any diagram as
 shown,
 \begin{center}
  \begin{tikzcd}[ampersand replacement=\&]
  {A} \arrow[d,rightarrowtail,"i"'] \arrow[r,"f"] \&
  {L} \arrow[d,twoheadrightarrow,"q"] \\
  {B} \arrow[r,"g"'] \&
  {M}
  \end{tikzcd}
 \end{center}
 we let $H(f,g)$ denote the set of lifts, i.e., maps $h\colon B\xrightarrow{}L$ such that $qh=g$
 and $hi=f$.  Then there is a naturally defined extension
 $K\xrightarrow{}T(f,g)\xrightarrow{}C$ such that splittings of $T(f,g)$ biject with
 $H(f,g)$.  In particular, if $\Ext^1_{\cat{A}}(C,K)=0$, then $H(f,g)$ is
 always nonempty, so $i$ has the left lifting property against
 $q$. 
\end{Prop}
\begin{proof}
 Consider the following diagram:
 \begin{center}
  \begin{tikzcd}[ampersand replacement=\&]
  {0}
   \arrow[r]
   \arrow[d] \&
  {A}
   \arrow[r,=]
   \arrow[d,rightarrowtail,"{\bsm i\\ f\esm}"] \&
  {A}
   \arrow[d,rightarrowtail,"{i}"] \\
  {L}
   \arrow[r,rightarrowtail,"{\bsm 0\\ 1\esm}"]
   \arrow[d,twoheadrightarrow,"{q}"'] \&
  {B\oplus  L}
   \arrow[r,twoheadrightarrow,"{\bsm 1 & 0\esm}"] 
   \arrow[d,twoheadrightarrow,"{\bsm -g & q\esm}"] \&
  {B}
   \arrow[d] \\
  {M}
   \arrow[r,=] \&
  {M}
   \arrow[r] \&
  {0.}
  \end{tikzcd}
 \end{center}
 Each column can be regarded as a complex, so the whole diagram is a
 short exact sequence of complexes, leading to a long exact sequence
 of homology groups.  This long exact sequence has only three non-zero
 terms, so it gives a short exact sequence of the form
 $K\xrightarrow{k}T\xrightarrow{r}C$, where $T=T(f,g)$ is the unique homology group of
 the middle column.  

 If $h\in H(f,g)$, then consider the following diagram:
 \begin{center}
  \begin{tikzcd}[ampersand replacement=\&]
  {0}
   \arrow[d] \&
  {A}
   \arrow[l]
   \arrow[d,rightarrowtail,"{\bsm i\\ f\esm}"] \&
  {A}
   \arrow[l,=]
   \arrow[d,rightarrowtail,"{i}"] \\
  {L}
   \arrow[d,twoheadrightarrow,"{q}"'] \&
  {B\oplus  L}
   \arrow[l,twoheadrightarrow,"{\bsm -h & 1\esm}"'] 
   \arrow[d,twoheadrightarrow,"{\bsm -g & q\esm}"] \&
  {B}
   \arrow[l,rightarrowtail,"{\bsm 1\\ h\esm}"']
   \arrow[d] \\
  {M} \&
  {M}
   \arrow[l,=] \&
  {0.}
   \arrow[l]
  \end{tikzcd}
 \end{center}
 The columns are the same complexes as before, and the horizontal maps
 give a splitting of our previous short exact sequence of complexes,
 and so induce a splitting of the homology group $T$.  

 For the opposite correspondence, we need more information about $T$.
 Let $Z$ be the corresponding group of cycles, which is the kernel of
 the map $(-g,q)\colon B\oplus  L\xrightarrow{}M$, or in other words, the pullback of
 $q$ and $g$.  We name the maps in the pullback square as follows:
 \begin{center}
  \begin{tikzcd}[ampersand replacement=\&]
  {Z}
   \arrow[r,"{\tg}"]
   \arrow[d,twoheadrightarrow,"{\tq}"'] \&
  {L}
   \arrow[d,twoheadrightarrow,"{q}"]  \\
  {B}
   \arrow[r,"{f}"'] \&
  {M.}
  \end{tikzcd}
 \end{center}
 Thus $\tg$ and $\tq$ are just the projections $B\oplus  L\xrightarrow{}L$ and
 $B\oplus  L\xrightarrow{}B$, restricted to $Z$.

 The differential in our complex is the map $\ti\coloneqq (i,f)\colon A\xrightarrow{}Z$, so
 $T$ is by definition the cokernel of $\ti$; we write $\tp\colon Z\xrightarrow{}T$
 for the quotient map.  We write $\tk\coloneqq (0,j)\colon K\xrightarrow{}Z$.  One checks
 that the following diagram 
 commutes:
 \begin{center}
  \begin{tikzcd}[ampersand replacement=\&]
  \& {K}
   \arrow[r,=]
   \arrow[d,rightarrowtail,"{\tk}"'] \&
  {K}
   \arrow[d,rightarrowtail,"{k}"] \\
  {A}
   \arrow[d,=]
   \arrow[r,rightarrowtail,"{\ti}"] \&
  {Z}
   \arrow[r,twoheadrightarrow,"{\tp}"]
   \arrow[d,twoheadrightarrow,"{\tq}"'] \&
  {T}
   \arrow[d,twoheadrightarrow,"{r}"]  \\
  {A}
   \arrow[r,rightarrowtail,"{i}"'] \&
  {B}
   \arrow[r,twoheadrightarrow,"{p}"'] \&
  {C.}
  \end{tikzcd}
 \end{center}
 We also see (by inspection of definitions and diagram chasing) that
 all rows and columns are exact, and that the bottom right square is a
 pullback. 

 Now suppose we are given a splitting of the sequence
 $K\xrightarrow{k}T\xrightarrow{r}C$, given by a map $n\colon C\xrightarrow{}T$ with $rn=1$.  By
 the pullback property, there is a unique map $\tn\colon B\xrightarrow{}Z$ with
 $\tq\tn=1_B$ and $\tp\tn=np\colon B\xrightarrow{}T$.  We claim that the map
 $h\coloneqq\tg\tn\colon B\xrightarrow{}L$ lies in $H(f,g)$.  Indeed, we first have
 $qh=q\tg\tn=g\tq\tn=g$.  Next, one checks that $\tp(\ti-\tn i)=0$ and 
 $\tq(\ti-\tn i)=0$ so the pullback property tells us that
 $\tn i=\ti\colon A\xrightarrow{}Z$.  This gives $hi=\tg\ti=f$ as required.

 We leave it to the reader to check that the above constructions are
 mutually inverse.
\end{proof}

\begin{Thm}\label{thm-proj-model-structure}\leavevmode
 Let $\cat U$ be a family satisfying \cref{hyp-basic}. Then \cref{defn-model-structure} gives a proper monoidal model
 structure on $\Ch(\A{\cat{U}})$.  In more detail:
 \begin{itemize}
  \item[(a)] Given morphisms $X\xrightarrow{f}Y\xrightarrow{g}Z$ in
   $\Ch(\A{\cat{U}})$, if two of $\{f,g,gf\}$ are weak equivalences
   then so is the third.
  \item[(b)] The classes $\we$, $\cof$, $\acf$, $\fib$ and $\afb$ are
   all closed under retracts and under composition.
  \item[(c)] Every cofibration has the left lifting property against
   every acyclic fibration, and every acyclic cofibration has the left
   lifting property against every fibration.
  \item[(d)] Every morphism in $\Ch(\A{\cat{U}})$ can be factored as a
   cofibration followed by an acyclic fibration, or as an acyclic
   cofibration followed by a fibration.
  \item[(e)] If $\unit'\to\unit$ is a cofibrant replacement, then the resulting map $\unit'\otimes X\to\unit\otimes X=X$ is a weak equivalence for all $X$.
  \item[(f)] Suppose that $f\colon U\to V$ and $g\colon X\to Y$ are
   cofibrations.  Let $P$ be the pushout of the maps 
   $V\otimes X\xleftarrow{f\otimes 1}U\otimes X
    \xrightarrow{1\otimes g}U\otimes Y$, and let $h$ be the
   induced map from $P$ to $V\otimes Y$.  Then $h$ is a cofibration, which
   is acyclic if either $f$ or $g$ is acyclic.
  \item[(g)] The pullback of a weak equivalence along a fibration is a
   weak equivalence, and the pushout of weak equivalence along a
   cofibration (or more generally, a monomorphism) is also a weak
   equivalence. 
 \end{itemize}
\end{Thm}

\begin{proof}\leavevmode
 \begin{itemize}
  \item[(a)] Clear from the definition.
  \item[(b)] Let $\Arr(\Ch(\A{\cat{U}}))$ be the category of
   arrows in $\Ch(\A{\cat{U}})$.  Suppose that
   $f,g\in \Arr(\Ch(\A{\cat{U}}))$ and that $g$ is a retract of
   $f$.  This means that $F(g)$ is a retract of $F(f)$ for any functor
   $F\colon \Arr(\Ch(\A{\cat{U}})) \to\cat{C}$. This applies in particular to the
   functors
   $\ker,\cok\colon \Arr(\Ch(\A{\cat{U}})) \to\Ch(\A{\cat{U}})$
   and
   $H_*\colon \Arr(\Ch(\A{\cat{U}})) \to \Arr(\Gr(\A{\cat{U}}))$.
   Thus we obtain that the five classes of morphisms are
   closed under retracts, using for $\acf$ that the class of contractible complexes of projectives is closed under retracts.
   It is clear that $\we$ and $\fib$ are
   closed under composition, and thus that $\afb$ is also closed under
   composition.  Now suppose we have cofibrations
   $X\xrightarrow{f}Y\xrightarrow{g}Z$.  Standard theory of abelian
   categories provides a short exact sequence
   $\cok(f)\to\cok(gf)\to\cok(g)$ in $\Ch(\A{\cat{U}})$.  As $f$ and
   $g$ are cofibrations, the objects $\cok(f)$ and $\cok(g)$ are
   degreewise projective, so
   $\cok(gf)\cong \cok(f)\oplus \cok(g)\in\Gr(\A{\cat{U}})$, and it
   follows that $gf$ is also a cofibration.  We also see (as in
   \cref{rem-triangulation}) that the sequence
   $\cok(f)\to\cok(gf)\to\cok(g)$ gives a distinguished triangle in
   $\K{\A{\cat{U}}}$.  If $f$ and $g$ are acyclic
   cofibrations, then $\cok(f)\simeq\cok(g)\simeq 0$ in
   $\K{\A{\cat{U}}}$, so $\cok(gf)\simeq 0$ in
   $\K{\A{\cat{U}}}$, so $gf$ is again an acyclic cofibration.
  \item[(c)] Suppose that $f\in\cof$ and $g\in\afb$, so $f$ is a
   monomorphism with $\cok(f)\in\Ch(\P{\cat U})$ and $g$ is an epimorphism
   with $\ker(g)$ acyclic.  \Cref{cor-no-ext} tells us
   that $\Ext^1_{\Ch(\A{\cat U})}(\cok(f),\ker(g))=0$, so
   \cref{prop-lift-abelian} tells us that $f$ has the left
   lifting property against $g$.  Now suppose instead that
   $f\in\acf$ and $g\in\fib$.  This means that $\cok(f)$ is a
   contractible complex in $\Ch(\P{\cat{U}})$, so it is projective in
   $\Ch(\A{\cat{U}})$ by \cref{cor-contr-proj}.  This means that
   $\Ext^1_{\Ch(\A{\cat U})}(\cok(f),\ker(g))=0$, so
   \cref{prop-lift-abelian} again tells us that $f$ has the
   left lifting property against $g$.
  \item[(d)] This is \cref{prop-factor-M} and \cref{prop-factor-N}.
  \item[(e)] This is immediate from \cref{lem-flat}.  (Moreover, in
   most cases $\unit$ will itself be cofibrant.)
  \item[(f)] Suppose that $f\colon U\to V$ and $g\colon X\to Y$ are
   cofibrations.  This means that in $\Gr(\A{\cat{U}})$ we can identify $f$
   and $g$ with inclusions $U\to U\oplus  A$ and $X\to X\oplus  B$,
   where $A,B\in\Gr(\P{\cat{U}})$.  Therefore, the resulting map 
   $h\colon P\to V\otimes Y$ is just the inclusion 
   \[ U\otimes X\oplus  A\otimes X\oplus  U\otimes B \to 
      U\otimes X\oplus  A\otimes X\oplus  U\otimes B \oplus  A\otimes B.
   \]
   This means that $h$ is a cofibration, with
   $\cok(h)\cong \cok(f)\otimes\cok(g)$.  (Here we have used the fact that
   $\P{\cat{U}}$ is closed under tensor products, see \cref{prop-proj-closed-under-tensor}.)  If either of $f$ or $g$ are acyclic then either $\cok(f)$ or $\cok(g)$ is contractible, so $\cok(f)\otimes\cok(g)$
   is contractible so $h$ is an acyclic cofibration.
  \item[(g)] Let $\cat{B}$ be an abelian category.  Suppose we have a
   pullback square in $\Ch(\cat{B})$ as shown below, in which $q$ is an
   epimorphism and $f$ is a quasi-isomorphism.  
   \begin{center}
    \begin{tikzcd}
     U \arrow[r,"p"] \arrow[d,"g"'] & 
     X \arrow[d,"f"] \\
     V \arrow[r,twoheadrightarrow,"q"'] &
     Y.
    \end{tikzcd}
   \end{center}
   Standard theory of abelian categories shows that $p$ is also an
   epimorphism and that the induced map $\ker(p)\to\ker(q)$ is an
   isomorphism.  We therefore have an expanded diagram as follows in
   $\Ch(\cat{B})$, where both rows are short exact sequences:
   \begin{center}
    \begin{tikzcd}
     \ker(p) \arrow[d,"\simeq"'] \arrow[r,rightarrowtail] &
     U \arrow[r,twoheadrightarrow,"p"] \arrow[d,"g"'] & 
     X \arrow[d,"f"] \\
     \ker(q) \arrow[r,rightarrowtail] &
     V \arrow[r,twoheadrightarrow,"q"'] &
     Y.
    \end{tikzcd}
   \end{center}
   This gives rise to a ladder diagram of homology groups, in which
   each of the two rows is a long exact sequence.  As the map
   $\ker(p)\to\ker(q)$ is an isomorphism and $f$ is a
   quasi-isomorphism, two of every three vertical maps are
   isomorphisms.  It follows by the Five Lemma that $g$ is also a
   quasi-isomorphism.  In other words, the pullback of a
   quasi-isomorphism along an epimorphism is again a quasi-isomorphism.
   By applying this to the dual category, we see that the pushout of a
   quasi-isomorphism along a monomorphism is also a quasi-isomorphism.
   In $\Ch(\A{\cat{U}})$, the fibrations are precisely the epimorphisms,
   and the cofibrations are a subclass of the monomorphisms, so
   claim~(g) follows.\qedhere
 \end{itemize}
\end{proof}

\subsection{Cofibrant generation}

We now verify that the model structure of
\cref{thm-proj-model-structure} is cofibrantly generated. Recall the functor $\Delta$ from \cref{defn-Dl}.
\begin{Def}\label{defn-cof-gens}
 We write $\gcof$ for the class of maps in $\Ch(\A{\cat{U}})$
 \[
 \gcof=\{j\colon \Sigma^{n-1}e_G\to\Sigma^n\Delta(e_G) \mid n\in\Z, \; G\in\cat{U}\}.
 \]
  We also write $\gacf$ for the set of all inclusions
 \[
 \gacf=\{0\to\Delta(\Sigma^ne_G)\mid n\in\Z, \; G\in\cat{U} \}.
 \]
  We write $\cof'$ for the smallest class of maps
 in $\Ch(\A{\cat{U}})$ such that 
 \begin{itemize}
  \item[(a)] $\gcof\subseteq\cof'$;
  \item[(b)] $\cof'$ is closed under pushouts, retracts and composites;
  \item[(c)] for any sequence $X(0)\to X(1)\to X(2)\to\dotsb$ with
   colimit $X(\infty)$, if all the maps $X(i)\to X(i+1)$ lie in
   $\cof'$, then so does the map $X(0)\to X(\infty)$.
 \end{itemize}
 We define $\acf'$ in the same way starting with $\gacf$.
\end{Def}

\begin{Lem}\label{lem-cof-gen}
 $\cof'=\cof$ and $\acf'=\acf$.
\end{Lem}
\begin{proof}
 It is easy to see that $\cof$ has the closure properties required for
 $\cof'$, so $\cof'\subseteq\cof$.  If we think of $\acf$ as $\we\cap\cof$,
 then it is also easy to see that $\acf$ has the closure properties
 required for $\acf'$, so $\acf'\subseteq\acf$.  

 Next, let $X$ be any object of $\Gr(\P{\cat{U}})$ (so $X\in\Ch(\P{\cat{U}})$
 with $d=0$). Then $X$ is a retract of a direct sum
 of terms like $\Sigma^le_G$ (see \cref{rem-projective-objects}), so the map $j\colon \Sigma^{-1}X\to\Delta(X)$ lies in
 $\cof'$ and the map $0\to\Delta(X)$ lies in $\acf'$.  The map $0\to X$
 is the pushout of $j$ along the map $\Sigma^{-1}X\to 0$ and so lies in
 $\cof'$.  The map $0\to\Delta(X)$ is the composite
 $0\to\Sigma^{-1}X\xrightarrow{j}\Delta(X)$ and so also lies in $\cof'$.  This
 proves that $\gacf\subseteq\cof'$, and it follows that $\acf'\subseteq\cof'$.

 Now consider an arbitrary acyclic cofibration $f\colon X\to Y$.  This is a
 monomorphism whose cokernel $Z$ is an acyclic complex of projectives.
 This is contractible and projective in $\Ch(\A{\cat{U}})$ by
 \cref{cor-acyc-proj}, so $f$ is just the inclusion 
 $X\to X\oplus Z$.  Also, \cref{prop-im-Dl} tells us that
 $Z\cong \Delta(V)$ for some $V\in\Gr(\P{\cat{U}})$, so $f$ is just the
 pushout of $0\to\Delta(V)$ along $0\to X$, so $f\in\acf'$.  This proves
 that $\acf'=\acf$.

 Now consider instead an arbitrary cofibration $f\colon X\to Y$, so $f$ is
 a monomorphism whose cofibre $Z$ lies in $\Ch(\P{\cat{U}})$; we must show that
 $f\in\cof'$.  Let $X(s)$ be the preimage in $Y$ of
 $L_{\leq s}Z\leq Z$, so $X\cong X(0)\leq X(1)\leq\dotsb \leq Y$
 and $Y$ is the colimit of the $X(s)$.  It will therefore suffice to
 check that the map $X(s-1)\to X(s)$ lies in $\cof'$ for all $s>0$.
 This map is clearly a monomorphism with cokernel $L_sZ$, which is
 $s$-pure.  Using \cref{lem-semisimple-complex}
 and \cref{prop-im-Dl}, we can split $L_sZ$ as $U\oplus \Delta(V)$, where
 $U$ has zero differential.  Here $U$ is projective in
 $\Gr(\A{\cat{U}})$ and $\Delta(V)$ is projective in $\Ch(\A{\cat{U}})$ so we can
 choose lifts giving a splitting $X(s)\cong X(s-1)\oplus  U\oplus \Delta(V)$
 where $\Delta(V)$ is a subcomplex and $d(U_{n+1})\leq X(s-1)_n$.  The
 differential can be thought of as giving a chain map
 $u\colon \Sigma^{-1}U\to X(s-1)$.  The inclusion $X(s-1)\to X(s)$ is the
 composite of the inclusions
 $X(s-1)\to X(s-1)\oplus  U\to X(s-1)\oplus  U\oplus \Delta(V)$.  The first
 of these is the pushout of $j\colon \Sigma^{-1}U\to\Delta(U)$ along $u$, so it
 lies in $\cof'$.  The second is the pushout of $0\to\Delta(V)$ along
 $0\to X(s-1)\oplus  U$, so it lies in $\acf'\subseteq\cof'$.  Thus, the map
 $X(s-1)\to X(s)$ lies in $\cof'$ as required.
\end{proof}

\begin{Prop}\label{prop-cof-gen}
 The model structure on $\Ch(\A{\cat{U}})$ is cofibrantly generated, with
 $\gcof$ as a generating set of cofibrations, and $\gacf$ as a
 generating set of acyclic cofibrations.
\end{Prop}
\begin{proof}
 Translating~\cite[Definition 2.1.17]{Hovey} into our notation, it
 will be enough to prove the following:
 \begin{itemize}
  \item[(a)] For all $n$ and $G$, the functor
   $\Hom_{\Ch(\A{\cat{U}})}(\Sigma^{n-1}e_G,-)$ preserves all filtered colimits
   (i.e., $\Sigma^{n-1}e_G$ is small).
  \item[(b)] The fibrations are precisely the maps that have the right
   lifting property against $\gacf$.
  \item[(c)] The acyclic fibrations are precisely the maps that have
   the right lifting property against $\gcof$.
 \end{itemize}
 Claim~(a) is clear from the fact that
 $\Hom_{\Ch(\A{\cat{U}})}(\Sigma^{n-1}e_G,U)\cong Z_{n-1}(U(G))$.  For~(b), a map $p\colon X\to Y$ has right lifting against $0\to\Sigma^n\Delta(e_G)$ if and only if the map $X_n(G)\to Y_n(G)$ is surjective, and thus $p$ has right lifting against all of $\gacf$ if and only if it is an epimorphism.

 We now consider~(c).  Let $p\colon X\to Y$ be a morphism in $\Ch(\A{\cat{U}})$.
 If $p$ is an acyclic fibration, we showed already in
 \cref{thm-proj-model-structure} that $p$ has right lifting against
 all of $\cof$ and in particular against $\gcof\subset\cof$.
 
 Conversely,
 suppose that $p$ has right lifting against $\gcof$. By \cref{lem-cof-gen}, $p$ also has right lifting against
 $\gacf$ and so is an epimorphism by part (b). The map $\ker(p)\to 0$ is the pullback of $p$ along $0\to Y$ and so has right lifting against $\gcof$. Right lifting against $\Sigma^{n-1}e_G\to\Sigma^n\Delta(e_G)$ translates directly to the statement that every $(n-1)$-cycle in $\ker(p)(G)$ is a boundary, so $\ker(p)$ is acyclic.
\end{proof}

\bibliographystyle{alpha}
\bibliography{reference}

\newcommand{\etalchar}[1]{$^{#1}$}
\begin{thebibliography}{BBP{\etalchar{+}}25}

\bibitem[AFH]{AvramovFoxbyHalperin03}
Luchezar~L. Avramov, Hans-Bj{\o}rn Foxby, and Stephen Halperin.
\newblock Differential graded homological algebra.
\newblock Preprint, 2003.

\bibitem[BBP{\etalchar{+}}25]{BBPSWttgeometry}
Miguel {Barrero}, Tobias {Barthel}, Luca {Pol}, Neil {Strickland}, and Jordan {Williamson}.
\newblock {The spectrum of global representations for families of bounded rank and VI-modules}.
\newblock {\em arXiv e-prints}, page arXiv:2506.21525, June 2025.

\bibitem[BMR14]{sixmodels}
Tobias Barthel, J.~P. May, and Emily Riehl.
\newblock Six model structures for {DG}-modules over {DGA}s: model category theory in homological action.
\newblock {\em New York J. Math.}, 20:1077--1159, 2014.

\bibitem[CEF15]{FI2}
Thomas Church, Jordan~S. Ellenberg, and Benson Farb.
\newblock F{I}-modules and stability for representations of symmetric groups.
\newblock {\em Duke Math. J.}, 164(9):1833--1910, 2015.

\bibitem[CEFN14]{FI1}
Thomas Church, Jordan~S. Ellenberg, Benson Farb, and Rohit Nagpal.
\newblock F{I}-modules over {N}oetherian rings.
\newblock {\em Geom. Topol.}, 18(5):2951--2984, 2014.

\bibitem[CH02]{projclass}
J.~Daniel Christensen and Mark Hovey.
\newblock Quillen model structures for relative homological algebra.
\newblock {\em Math. Proc. Cambridge Philos. Soc.}, 133(2):261--293, 2002.

\bibitem[Cor]{MO}
Yves Cornulier.
\newblock Is every finite group the outer automorphism group of a finite group?
\newblock MathOverflow.
\newblock URL:https://mathoverflow.net/q/372563 (version: 2024-05-08).

\bibitem[Hir03]{Hirschhorn}
Philip~S. Hirschhorn.
\newblock {\em Model categories and their localizations}, volume~99 of {\em Mathematical Surveys and Monographs}.
\newblock American Mathematical Society, Providence, RI, 2003.

\bibitem[Hov99]{Hovey}
Mark Hovey.
\newblock {\em Model categories}, volume~63 of {\em Mathematical Surveys and Monographs}.
\newblock American Mathematical Society, Providence, RI, 1999.

\bibitem[Lur17]{HA}
Jacob Lurie.
\newblock Higher algebra.
\newblock 1553~pages, available from the author's website, 2017.

\bibitem[Nag19]{Nagpal}
Rohit Nagpal.
\newblock V{I}-modules in nondescribing characteristic, part {I}.
\newblock {\em Algebra Number Theory}, 13(9):2151--2189, 2019.

\bibitem[Nee92]{Neemanlocalization}
Amnon Neeman.
\newblock The connection between the {$K$}-theory localization theorem of {T}homason, {T}robaugh and {Y}ao and the smashing subcategories of {B}ousfield and {R}avenel.
\newblock {\em Ann. Sci. \'Ecole Norm. Sup. (4)}, 25(5):547--566, 1992.

\bibitem[PS17]{PutmanSam2017}
Andrew Putman and Steven~V. Sam.
\newblock Representation stability and finite linear groups.
\newblock {\em Duke Math. J.}, 166(13):2521--2598, 2017.

\bibitem[PS22]{PolStrickland2022}
Luca Pol and Neil~P. Strickland.
\newblock Representation stability and outer automorphism groups.
\newblock {\em Doc. Math.}, 27:17--87, 2022.

\bibitem[Sam24]{realizingout}
Benjamin Sambale.
\newblock Characterizing inner automorphisms and realizing outer automorphisms.
\newblock arXiv:2405.02992, 2024.

\bibitem[Sch18]{Schwedebook}
Stefan Schwede.
\newblock {\em Global homotopy theory}, volume~34 of {\em New Mathematical Monographs}.
\newblock Cambridge University Press, Cambridge, 2018.

\bibitem[Spa88]{Spaltenstein}
N.~Spaltenstein.
\newblock Resolutions of unbounded complexes.
\newblock {\em Compositio Math.}, 65(2):121--154, 1988.

\bibitem[{Sta}24]{stacks}
The {Stacks project authors}.
\newblock The stacks project.
\newblock \url{https://stacks.math.columbia.edu}, 2024.

\bibitem[Wim17]{Wimmerthesis}
Christian Wimmer.
\newblock Rational global homotopy theory and geometric fixed points.
\newblock PhD thesis, Universit{\"a}t Bonn, 2017.

\bibitem[Yek20]{Yekutieli}
Amnon Yekutieli.
\newblock {\em Derived categories}, volume 183 of {\em Cambridge Studies in Advanced Mathematics}.
\newblock Cambridge University Press, Cambridge, 2020.

\end{thebibliography}

\end{document}